\documentclass[12pt]{amsart}

\usepackage{graphicx}
\usepackage{geometry,enumitem}
\usepackage{hyperref}

\usepackage{color}
\usepackage{accents}

\usepackage{amsthm, amsxtra}

\numberwithin{equation}{section}

\newtheorem{theorem}{Theorem}[section]
\newtheorem{theorem*}{Theorem}
\newtheorem{remark}[theorem]{Remark}
\newtheorem{lemma}[theorem]{Lemma}

\newtheorem{proposition}[theorem]{Proposition}
\newtheorem{corollary}[theorem]{Corollary}

\newtheorem*{question*}{Question}

\newtheorem{definition}[theorem]{Definition}

\newcommand{\nn}{\nonumber}
\newcommand{\R}{\mathbb{R}}
\newcommand{\E}{\mathbb{E}}
\newcommand{\e}{{\rm e}^}
\newcommand{\mb}{\mathbf}
\newcommand{\bs}{\boldsymbol}

\author{Nassif Ghoussoub, Young-Heon Kim and Aaron Zeff Palmer}

\address{Department of Mathematics\\ University of British Columbia\\ Vancouver, V6T 1Z2 Canada}

\email{nassif@math.ubc.ca, yhkim@math.ubc.ca, azp@math.ubc.ca}

\title{PDE Methods for Optimal Skorokhod Embeddings}

\thanks{The   
authors are partially supported by  the 
Natural Sciences and Engineering Research Council of Canada (NSERC). \\
\copyright 2018 by the authors.
}

\date{\today}

\begin{document}

\begin{abstract}
	We consider cost minimizing stopping time solutions to Skorokhod embedding problems, which deal with transporting a source probability measure to a given target measure through a stopped Brownian process. PDEs and a free boundary problem approach are used to address the problem in general dimensions with space-time inhomogeneous costs given by Lagrangian integrals along the paths.  We introduce an Eulerian---mass flow---formulation of the problem, whose dual is given by Hamilton-Jacobi-Bellman type variational inequalities.  Our key result is the existence (in a Sobolev class) of optimizers for this new dual problem, which in turn determines a free boundary, where the optimal {\em Skorokhod transport} drops the mass in space-time. This complements and provides a constructive PDE alternative to recent results of Beiglb\"ock, Cox, and Huesmann, and is a first step towards developing a general optimal mass transport theory involving mean field interactions and noise.  
\end{abstract}

\maketitle
\tableofcontents

\section{Introduction} \label{sec:introduction}

	Given two Radon probability measures $\mu$ and $\nu$ on $\R^d$, the Skorokhod embedding problem consists of constructing a stopping time $\tau$ such that $\nu$ is realized by the distribution of $B_\tau$ (i.e, $B_\tau \sim\nu$ in our notation), where $B_t$ is Brownian motion starting with $\mu$ as a source distribution, i.e., $B_0\sim \mu$. We shall denote by ${\mathcal T}(\mu, \nu)$ the set of such --possibly randomized-- stopping times with finite expectation. If $\mu$, $\nu$ are supported in an open convex subset $O \subset \R^d$, then ${\mathcal T}_O(\mu, \nu)$ denotes those in ${\mathcal T}(\mu, \nu)$ such that $\tau \leq \tau_O:=\inf\{t;\ B_t\not\in O\}$. The martingale property of Brownian motion obviously imposes a natural necessary condition on the pair $\mu$, $\nu$ for the existence of such a stopping time, namely that they should be {\em in subharmonic order}, denoted $\mu\prec \nu$, which means that 
	\begin{equation}\label{eqn:subharmonic_order}
		\int_{ O} h\, d\mu\leq \int_{ O} h \, d\nu \quad \hbox{
	 		for all $h$ smooth and subharmomic on $O$.}     
	\end{equation} 
	The fact that condition (\ref{eqn:subharmonic_order}) is also sufficient to guarantee that ${\mathcal T}_O(\mu, \nu)$ is non-empty has been the subject of a large number of papers. Indeed, and as reported by Ob\l{}oj \cite{obloj2004skorokhod}, one can find more than two dozen constructions of such embeddings in the literature ever since Skorokhod \cite{skorokhod1965studies} gave the original solution in the one-dimensional case. Among those, there are some with additional interesting properties arising from the fact that they optimize various functionals involving stopped Brownian motion. Indeed, Rost  had established in \cite{Rost} that Root's embedding \cite{Root} minimizes the variance of the stopping time (or equivalently $\mathbb{E}[\tau^2]$), which was conjectured by Kiefer \cite{kiefer}.   
	He also established a new embedding that minimizes $\mathbb{E}[\tau^p]$ for $0<p<1$.   There has been an attempt to relate the Root's embedding problem to obstacle problems by Cox and Wang \cite{cox2013root} and Gassiat, Oberhauser, and dos Reis \cite{gassiat2015root}, but, their works were restricted to one-dimension: see our Remark \ref{rem:potential}.

This led Beiglb\"ock-Cox-Huesmann \cite{beiglboeck2017optimal} to consider the more general problem of identifying and characterizing solutions of the Skorokhod embedding problem that optimize various functionals of stopped Brownian motion. They also argued that if one thinks of a stopping time $\tau$ as ``a transport plan" from Wiener measure starting at $\mu$  to the target measure $\nu$, then the analogy with the theory of mass transport --though not directly applicable-- could provide a powerful intuition towards developing an analogous theory. They do so in \cite{beiglboeck2017optimal} by using stochastic analysis to essentially extend the measure theoretic duality methods \`a la Kantorovich and the monotonicity characteristics of optimal mass transport plans to this setting. 

Our results can be seen as complementary to theirs in several ways: they involve a new important Eulerian formulation, its crucial dual form, which allow us to use PDE methods and the theory of Hamilton-Jacobi-Bellman inequalities. The optimal stopping time we obtain, will be hitting times of corresponding free boundaries. An important part of our approach is our proof of the attainment in the new dual problem, which has been an elusive issue in martingale transport theory, especially in higher dimensions \cite{ghoussoub2015structure} (see also the following related results \cite{de2017irreducible}, \cite{de2018quasi}).

	We consider two Radon probability measures $\mu$, $\nu$ with finite expectations,  whose support lie in a given convex domain $O\subset \R^d$, and in subharmonic order on $O$. Our primal problem will be the following minimization: 
	\begin{align} \label{eqn:Skorokhod_cost}
		{\mathcal P}_0(\mu,\nu) :=  \inf
		\Big\{\mathbb{E}\Big[ \int_0^\tau L(t,B_t)dt\Big];\ \tau \in {\mathcal T}_O(\mu, \nu)
		\Big\}. 
	\end{align}
	These Lagrangian costs are a special case of the general costs considered in \cite{beiglboeck2017optimal}, where they prove a general duality theorem, but also give an extension of the result of \cite{beiglboeck2017optimal}, where it is shown the optimizers are given by the Root and Rost embeddings in the case the Lagrangian $L(t, x)$ is only a (strictly increasing/decreasing) function of time, for which we will handle dependence on the spatial variable.
	We note that costs that depend on the end time and position can be reduced to Lagrangian costs by It\^{o}'s formula
	$$
		\mathbb{E}\big[g(\tau,B_\tau)\big]= \mathbb{E}\Big[\int_0^\tau L^g(t,B_t)dt\Big],
	$$
	where $L^g(t,x)=\partial_t g(t,x)+\frac{1}{2}\Delta g(t,x))$.
	In a forthcoming paper \cite{GKP2}, we shall deal with Lagrangians defined on phase space (see the end of this introduction). Throughout this article,  the Lagrangian $L$ will be assumed to satisfy:
	\begin{enumerate} [label=\textbf{(H\arabic*)}] \setcounter{enumi}{-1}
		\item \label{itm:continuous}
		 $L$ is non-negative and belongs to $C_{-\gamma}(\R^+\times \overline{O})$, the latter being the space of continuous functions $w$ on $\R^+\times \overline{O}$ with 
		\begin{align*}
 			\hbox{$ \e{-\gamma t} w(t,x) \rightarrow 0$ as $t\to \infty$,   uniformly in $x$.}
		\end{align*}
		Here we fix $\gamma$ satisfying $0<\gamma<\lambda$, where $\lambda$ is the Poincar\'{e} constant of $O$, and we denote $\R^+ = \{ t \in \R | \ t \ge 0\}$. 
	\end{enumerate}
	 	Note that the set of randomized stopping times ${\mathcal T}_O(\mu, \nu)$ is non-empty, convex and compact in an appropriate topology, and therefore the existence of a minimum is not really a problem. The challenge is to characterize such solutions and to show when they are natural and unique stopping times, preferably characterized as hitting times of certain barrier sets that can be naturally identified from the Lagrangian as well as the source and target measures.

	Duality plays an important role in these problems and Beiglb\"ock et al.\ had considered in \cite{beiglboeck2017optimal}  the following dual problem  to \eqref{eqn:Skorokhod_cost}: 
	{\small
	\begin{align}\label{eqn:probabilistic_dual}
 		{\mathcal D}_0(\mu,\nu):=\sup_{\psi\in C(\overline{O}),G\in \mathcal{K}^+_{-\gamma}}\Big\{\int_{\overline{O}}\psi(z)d\nu-\E\big[G_0\big];\ G_t-\psi(B_t)\geq -\int_0^t L(s,B_s)ds\Big\}, 
 	\end{align}
	}
	\hspace{-0.3em}where $\mathcal{K}^+_{-\gamma}$ denotes the set of continuous supermartingales on the probability space of Brownian motion with $\gamma$-exponential growth. In particular, if $G\in \mathcal{K}^+_{-\gamma}$, then $\omega \mapsto G$ is continuous with respect to the topology on $\Omega=C(\R^+;\R^d)$ given by uniform convergence of paths.

	That ${\mathcal P}_0(\mu,\nu)={\mathcal D}_0(\mu,\nu)$ was essentially  proven in \cite{beiglboeck2017optimal} for a slightly different problem. We shall also include a proof in the appendix, Theorem \ref{thm:probabilistic_weak_duality}. 		
	
	\subsection{Eulerian formulation and its dual}
	Our analysis hinges on understanding two other related 
	problems.  One consists of an Eulerian formulation of the primal problem.  We give here a heuristic description leaving the appropriate function spaces to be defined in the next section. For a stopping time $\tau$ such that $B_\tau \sim \nu$, we consider the space-time distribution  of stopped particles ${\rho}$, that is $(\tau,B_\tau)\sim\rho$, where the latter, which we call {\em the stopping measure},  is a probability measure on $\R^+\times \overline{O}$. Note that $\rho$ has a target distribution with  spatial marginal $\nu$, that is
	\begin{align} \label{eqn:Eulerian_target} 
 			\int_{\overline{O}} u(z)\, \nu(dz) =&\ \int_{\overline{O}}\int_{\R^+} u(x){\rho}(dt,dx)\quad \hbox{for all $u\in C(\overline{O})$.}	
 	\end{align}
	Let now  $\eta$ be a non-negative measure on $\R^+\times \overline{O}$ that corresponds to the distribution of Brownian motion before it has stopped. It can be expressed by the following   evolution equation in its `very weak' form via smooth test functions $w$:
	\begin{align}\label{eqn:Skorokhod_evolution}
		-\int_{\overline{O}}w(0,y)\, \mu(dy)=&\ \int_{\overline{O}}\int_{\R^+}\Big[\frac{\partial}{\partial t}w(t,x)+\frac{1}{2}\Delta w(t,x)\Big]\eta(dt,dx)\\
		&\ -\int_{\overline{O}}\int_{\R^+} w(t,x)\rho(dt,dx).\nn
	\end{align}
	Note that if ${\rho}$ has the density with respect to $\eta$, then $\frac{d\rho}{d\eta}(t,x)$ is the conditional probability to stop given $B_t=x$. 
	
 	We shall  say that $(\eta,\rho)$ is {\em an admissible pair}, provided they satisfy (\ref{eqn:Eulerian_target}) and (\ref{eqn:Skorokhod_evolution}) as well as an additional condition of exponential decay in time -to be defined in the next section. We shall then consider the following linear problem 
  	\begin{align} \label{eqn:Eulerian_cost} 
 		{\mathcal P}_1(\mu,\nu) = \inf_{\eta,{\rho}}\Big\{\int_{\overline{O}}\int_{\R^+} L(t,x)\eta(dt,dx);\ (\eta,{\rho})\ {\rm is\ admissible}\Big\}. 
 	\end{align}
	We shall then prove that it is equivalent to the original primal problem, that is
  	\begin{align}  
 		{\mathcal P}_0(\mu, \nu)=	{\mathcal P}_1(\mu,\nu).\nn  
 	\end{align}
	This will be proved by means of another dual problem  that is motivated by standard Monge-Kantorovich theory, once applied to a cost given by a Lagrangian function. This was considered -in a deterministic context- by Bernard-Buffoni \cite{B-B} for fixed end-times problems, and by the authors \cite{GKP} in the case when the end-times are free. Here we shall consider the following dual problem:
	\begin{align} \label{eqn:dual_cost}
		{\mathcal D}_1(\mu,\nu) := \sup_{\psi,J}\Big\{\int_{\R^d}\psi(z)\nu(dz) - \int_{\R^d}J(0,y)\mu(dy);\ V_\psi[J]\geq 0\Big\}, 
	\end{align}
	where $\psi:\R^d\rightarrow \R$ is continuous, $J:\R^+\times \R^d\rightarrow \R$ is $C^2$, and $V_\psi$ is the quasi-variational Hamilton-Jacobi-Bellman operator
	\begin{align} \label{eqn:variational_operator}
		V_\psi[J](t,x):=\min\left\{\begin{array}{r} J(t,x) -\psi(x)\\ -\frac{\partial}{\partial t}J(t,x)-\frac{1}{2}\Delta J(t,x)+L(t,x)\end{array}\right\}.
	\end{align}
	For each end-potential $\psi:\R^d\rightarrow \R$ that is continuous, the value-function, $J_\psi:\R^+\times \R^d\rightarrow \R$ is defined via the 
	dynamic programming principle
	\begin{align}\label{eqn:dynamic_programming}
		J_\psi(t,x) := \sup_{\tau \in \mathcal{R}^{t,x}}\Big\{\mathbb{E}^{t,x}\Big[\psi(B_\tau)-\int_t^\tau L(s,B_s)ds\Big]\Big\},
	\end{align}
	where the expectation superscripted with $t,x$ is with respect to the Brownian motions satisfying $B_t=x$, and the minimization is over all finite-expectation stopping times $\mathcal{R}^{t,x}$ on this restricted probability space such that $\tau \ge t$, which will be defined more precisely at the begininng of Section 2. Note $\tau$ is not necessarily required to be in $ \mathcal{T}_O(\mu, \nu)$.  
	The value function satisfies the following Hamilton-Jacobi-Bellman variational inequality:
	\begin{align*}  
		\left\{\begin{array}{rr}  &J_\psi(t,x)\geq \psi(x)\hfill  \\ 
			&\frac{\partial}{\partial t}J_\psi(t,x)+\frac{1}{2}\Delta J_\psi(t,x)\leq L(t,x) \hfill \\
			& \frac{\partial}{\partial t}J_\psi(t,x)+\frac{1}{2}\Delta J_\psi(t,x)= L(t,x)\quad \hbox	{if $J_\psi(t,x)> \psi(x)$}, 
 		\end{array}\right.
	\end{align*}	
	or equivalently $V_\psi[J_\psi]=0$ in the sense of viscosity. 
	Moreover,  
	\begin{align*}
		J_\psi (t,x)  = \inf \{ J (t, x)  \ ; \ V_\psi [J] \ge 0\}.
	\end{align*}
	
	In Section \ref{sec:duality}, we shall prove that (\ref{eqn:dual_cost}) is also dual to the primal problem	  (\ref{eqn:Skorokhod_cost}), that is  
	\begin{equation} 
		{\mathcal P}_0(\mu, \nu)={\mathcal P}_1(\mu,\nu)={\mathcal D}_0(\mu, \nu)={\mathcal D}_1(\mu,\nu). 
	\end{equation}	
	The equivalence of the two dual problems is through a connection with the so-called Snell problem \cite{el1997reflected}. Indeed, if one considers the stochastic process 
	$$Y_t=\psi(B_t)-\int_0^tL(s,B_s)ds,$$ then  
	$$G_t = J_\psi(t,B_t)-\int_0^tL(s,B_s)ds$$
	 is the smallest supermartingale such that $G_t\geq Y_t$. Moreover, under suitable hypothesis, the ``Snell stopping time," that is $\tau=\inf\{t \, ; \, G_t=Y_t\}$, which renders the process $G_{t\wedge \tau}$ a martingale, will  
	 coincide with the optimal time we are seeking.

\subsection{Dual attainment and verification}
	One of our key results -- proved in Section \ref{sec:attainment_verification} -- is that both  the Eulerian formulation of the primal problem ${\mathcal P}_1(\mu,\nu)$ and the dual  problem ${\mathcal D}_1(\mu,\nu)$ are attained  in Sobolev class. For that, we need additional assumptions. 
Indeed, we may restrict the minimization of $\mathcal{D}_1(\mu,\nu)$ to the pair $\psi$ and its value function $J_\psi$ of (\ref{eqn:dynamic_programming}).
	 Actually, for technical reasons, we shall restrict the problem to the closure of a bounded open convex set $\overline{O}$ and show that $V_\psi[J]=0$ may be posed with additional Dirichlet boundary conditions, i.e., $\psi(x)=J_\psi(t,x)=0$ for $x\in \partial O$. However, we prove the attainment of the dual problem in a Sobolev class of functions, such that $(\psi,J)$ satisfy $V_\psi[J]= 0$ in a weak sense. We show that $\psi$ can be taken to be lower semicontinuous, and in this context we can use results in \cite{bensoussan2011applications} that yield that the `minimal weak solution' to $V_\psi[J]= 0$ coincides with the value function of (\ref{eqn:dynamic_programming}).

Here are our assumptions on the source and target measures $\mu, \nu$:
 	\begin{enumerate} [label=\textbf{(S\arabic*)}] \setcounter{enumi}{-1}

		\item \label{itm:convex_support} The closure of the support of both $\mu$ and $\nu$ is contained in an open, bounded and convex set $O\subset \R^d$.  
		\item \label{itm:L2_initial}
			  $\mu$ and $\nu$ are absolutely continuous with respect to Lebesgue measure and have densities in $L^2(O)$. 
		\item \label{itm:subharmonic_order} $\mu \prec \nu$ for the subharmonic order on $O$.
	\end{enumerate}
	We shall require that the Lagrangian satisfies the following additional assumption:
	\begin{enumerate} [label=\textbf{(H\arabic*)}] \setcounter{enumi}{0}
	 	\item \label{itm:bounded} 
		$L(t,x)\leq D<\infty$ for all $(t,x)\in \R^+\times \overline{O}$.
	\end{enumerate}
The following result summarizes the linkages between the optimizers of the various primal and dual problems. 	
\begin{theorem}\label{thm:strong_duality_0}
		Suppose {\normalfont \ref{itm:convex_support}, {\normalfont \ref{itm:L2_initial}}, \ref{itm:subharmonic_order}}, {\normalfont   \ref{itm:continuous}} and {\normalfont \ref{itm:bounded}}. Then,
		\begin{enumerate} 
		\item The maximum of the Eulerian dual problem ${\mathcal D}_1(\mu, \nu)$ is attained at $\psi\in  H_0^1(O)$ and $J\in \mathcal{X}$  (see \eqref{eqn:X}).

	\item\label{(2)}  The infimum of the primal Eulerian problem ${\mathcal P}_1(\mu, \nu)$ is attained at an admissible pair $(\eta, \rho)\in L^2_\gamma(\R^+;H_0^1(O))\times \mathcal{X}^*$ (see also \eqref{eqn:l2gamma}), which satisfies the following complementary slackness condition:
	{\small
		\begin{align}
 	 		0=&\ \int_{{O}}\int_{\R^+}\Big[J(t,x)-\psi(x)\Big]{\rho}(dt,dx)\nn\\
 	 		&\ + \int_{\R^+}\int_{O}\Big[-\eta(t,x)\frac{\partial}{\partial t}J(t,x) +\frac{1}{2} \nabla J(t,x)\cdot\nabla \eta(t,x)+L(t,x)\eta(t,x)\Big] dxdt.\nn
 	 	\end{align}
}
	\item\label{(3)} Furthermore, $\psi$ is lower semicontinuous, and we may take $J$ to be
		$$
			J(t,x)=J_\psi(t,x):=\sup_{\tau \in \mathcal{R}^{t,x}}\Big\{\mathbb{E}^{t,x}\Big[\psi(B_\tau)-\int_t^\tau L(s,B_s)ds\Big]\Big\}.
		$$
	\end{enumerate}
	\end{theorem}
It is clear from \text{(2)} that the optimal process must  stop in the space-time set 
\begin{align*}F:=\{(t, x);\ J(t,x)=\psi(x)\},
\end{align*}
 i.e., the coincidence set of the equation $V_\psi[J]=0$ viewed as the obstacle problem for $J$ with obstacle $\psi$.  This enables us in Section~\ref{sec:hitting_times} to  
 address  cases where the optimal stopping time is natural and unique. Both properties will follow if the optimal stopping time is given by the hitting time of a barrier, which we establish under strict monotonicity assumptions on $L$. In this case, the barrier is the coincidence set $F$. For example, 
 if the Lagrangian $L$ strictly increases in time, then one can define the function 
$  s(x)=\inf\{t;J_\psi(t,x)=\psi(x)\}$ 
and the barrier 
$F=R_c:= \{(t,x);\ t\geq s(x)\}$, 
 so that the optimal time is the first hitting-time of $R_c$, that is $\tau= \inf\{t; (t,B_t)\in R_c\}.$

We consider in the final section  more explicit behavior of the free boundary, by deriving another equation for the function $s(x)$. We will also discuss connection of our analysis to a setting involving backward stochastic differential equations (BSDE). A couple of concrete examples are provided for demonstration. 
	Finally, a few additional material are given in  Appendix \ref{S:appendix}, including a proof of weak duality and various notions of solutions to variational inequalities and their properties that are used throughout the paper.
	
The methods in this paper can be extended to other processes, for instance when Brownian motion is replaced by a Markov process generated by a uniformly elliptic operator $L$.  In this case Laplacian, $\frac{1}{2}\Delta$, is replaced by $L$, and the subharmonic functions should be replaced by $L$-subharmonic; see for example \cite{bensoussan2011applications}.  We deal only with Brownian motion for simplicity and will address the following important extensions of our methods in subsequent work.

\subsection{When cost functions are not induced by a Lagrangian}
In a recent paper \cite{GKP3}, we show that the methods developed in this paper can be used to deal with other important higher dimensional cases, where the cost is a function of the initial and end points, i.e.\ $c(x,y) = f (|x-y|)$. In other words, with problems such as 	
	\begin{align} \label{eqn:classical.cost}
		{\mathcal P}_0(\mu,\nu) :=  \inf
		\Big\{\mathbb{E}\Big[c(B_0,B_\tau)\Big];\ \tau \in {\mathcal T}_O(\mu, \nu)
		\Big\}. 
	\end{align}
	For other related papers in 1 dimension, we refer to \cite{beiglbock2017monotone} and \cite{huesmann2018monotonicity}.

\subsection{Control and mean field extensions}

This paper plays an important role in a larger research program on dynamic optimization problems with target constraints.  Deterministic control problems have already been the subject of \cite{GKP} by the authors.  While here we study Brownian motion, we have in mind important more general controlled stochastic differential equations such as, 
$$
	dX_t = \alpha_t dt + dB_t.
$$
The Lagrangian will then be assumed to also depend on the drift. One can also incorporate mean field interactions through the density and stopping distribution.  Such dependence is natural in the Eulerian formulation where we consider the Fokker-Planck equation
$$
	\rho(t,x)+\partial_t \eta(t,x)+\nabla \cdot a(t,x) \eta(t,x)=\frac{1}{2}\Delta \eta(t,x),
$$
with the cost
$$
	\int_{\R^d}\int_{\R^+}\Big[L\big(t,x,a(t,x)\big)\eta(t,x)+F\big(\eta(t,x)\big)+G\big(\rho(t,x)\big)\Big]dtdx,
$$
and the target constraint remains $\int_{\R^+}d\rho = \nu$.
The resulting Hamilton-Jacobi-Bellman inequality appears as
\begin{align}
	J(t,x)-\psi(x)&\geq -G'\big(\rho(t,x)\big) &\ \forall\ (t,x)\nn\\
	J(t,x)-\psi(x)&= -G'\big(\rho(t,x)\big),&\ {\rm if}\ \rho(t,x)>0\nn\\
	-\partial_t J(t,x) - A\cdot \nabla J(t,x)-\frac{1}{2}\Delta J(t,x)&\geq -L(t,x,A)-F'(\eta(t,x)),&\ \forall\ (t,x,A)\nn\\
	-\partial_t J(t,x) - A\cdot \nabla J(t,x)-\frac{1}{2}\Delta J(t,x)&= -L(t,x,A)-F'(\eta(t,x)),&\ {\rm if}\ (t,x,A) \in {\mathcal S},\nn
\end{align}
where 
$${\mathcal S}:=\{(t,x,A); 
\hbox{$\eta (t, x) >0$ and $A$  maximizes $A\cdot \nabla J(t,x)-L(t,x,A)$}. 
$$
This will be the subject of a forthcoming paper \cite{GKP2}. \\

\noindent\textbf{Acknowledgment:} We are thankful to Martin Barlow and Ed Perkins for helpful comments and references. We also thank A. Cox and M. Huesmann for pointing out several papers related to this work, and an anonymous referee for pertinent comments and pointing out additional related work. 

\section{Duality, Eulerian Embeddings and Variational Inequalities} \label{sec:duality}

This section is devoted to the proof of the equalities,
\begin{equation}\label{eqn:dual_equalities}
		{\mathcal P}_0(\mu, \nu)={\mathcal P}_1(\mu,\nu)={\mathcal D}_0(\mu, \nu)={\mathcal D}_1(\mu,\nu). 
	\end{equation}		

\subsection{Eulerian embedding of the stopping problem.}
Formally, we let $B_t$ be a Brownian motion with  $B_0\sim \mu$ and we consider the randomized stopping times $\tau$ that satisfy $B_\tau \sim\nu$.  We consider the filtered probability space $(\Omega,\mathcal{F},\{\mathcal{F}_t\}_{t\in \R^+},\mathbb{P})$ where $\Omega =C(\R^+;\R^d)$, $\mathbb{P}$ the Wiener measure with initial distribution $\mu$, and $\{\mathcal{F}_t\}_{t\in \R^+}$ the natural filtration of the Brownian motion. We define the space of randomized stopping times with finite expectation as
{\small
	\begin{align}
		\mathcal{R}:=\Big\{\alpha: \Omega \rightarrow \mathcal{M}(\R^+)\, ; \, &\,  \alpha\geq 0, \,  \alpha(\R^+)=1, \, \int_{\R^+}t \, d\alpha<\infty \, \hbox{and  $\alpha([0,t])$ is   $\mathcal{F}_t$-measurable } \forall\, t\Big\}\nn 
	\end{align}
	}
	\hspace{-0.3em}where $\mathcal{M}(\R^+)$ denotes the space of Radon measures on $\R^+$. 
	We often abbreviate a randomized stopping time by $\tau\sim \alpha$ such that
	$$
		\mathbb{E}\big[f(\tau)\big]=\mathbb{E}\Big[\int_{\R^+}f(t)\alpha(dt)\Big].
	$$
	The condition $B_\tau\sim \nu$ is now equivalent to
	$$
		\mathbb{E} \big[g(B_\tau)\big] = \int_{\overline{O}} g(z)\nu(dz),\ \forall g\in C(\overline{O}).
	$$
	We will say that a subset $Q\subset \R^+\times \Omega$ is almost sure for a randomized stopping time if
	$$
		\mathbb{E}\big[\int_{\R^+} \mathbf{1}\{t\in Q\} d\alpha \big]=0
	$$
	where $\mathbf{1}$ is the indicator function of the set.  Often this will appear instead as an abbreviated form, i.e.\ $\mathbf{1}\{t\leq \eta\}:=\mathbf{1}\{t\in Q_\eta\}$ where $\eta$ is a stopping time and $Q_\eta=\{(t,\omega);\ t\leq \eta(\omega)\}$.

	For each $(t,x)$ we also consider a filtered probability space of Brownian motions beginning at $B_t=x$, and we use $\mathbb{E}^{t,x}$ to denote the expectation.  We let $\mathcal{R}^{t,x}$ denote the space of randomized stopping times, larger than or equal to $t$, with finite expectation on this probability space. 

	We consider $C_{-\gamma}(\R^+\times \overline{O})$ given as in \ref{itm:continuous}, i.e.\ the continuous functions satisfying
	\begin{align*}
 			\hbox{$ \e{-\gamma t} w(t,x) \rightarrow 0$ as $t\to \infty$,   uniformly in $x$,}
		\end{align*}
		where $\gamma$ is fixed such that $0<\gamma<\lambda$ for $\lambda$ the Poincar\'{e} constant of $O$.
		 This is a Banach space with the norm 
	$$
		\|w\|_{C_{-\gamma}(\R^+\times \overline{O})}:=\sup\big\{|\e{-\gamma t}w(t,x)|;\ (t,x)\in \R^+\times \overline{O}\big\}<\infty, 
	$$
	 whose dual space $\mathcal{M}_\gamma(\R^+\times \overline{O})$ 
	  is the finite (signed) Radon measures with $\gamma$-exponential decay.  
	We will also let $C_{-\gamma}^{1,2}(\R^+\times \overline{O})$ denote the functions whose first derivative in time and second derivatives in space lie in $C_{-\gamma}(\R^+\times \overline{O})$.
	
	We consider the `very weak' evolution equation for $\eta,\rho\in \mathcal{M}_\gamma(\R^+\times \overline{O})$, $\eta,\rho\geq 0$, that satisfy (\ref{eqn:Skorokhod_evolution}),	which is formally equivalent to 
	$$\hbox{$\eta\geq 0$ and $\partial_t\eta\leq \frac{1}{2}\Delta \eta$ with $\eta(0,\cdot)=\mu$ in the sense of distributions. }
	$$  
	The stopping measure  $\rho$ 
	 allows us to encode the target constraint as (\ref{eqn:Eulerian_target}), which is formally given by $
	 	\rho(\R^+,\cdot)=\nu.
	$
\begin{definition}
	We say that $(\eta,\rho)$ is admissible if $\eta\in \mathcal{M}_\gamma(\R^+\times\overline{O})$, $\eta\geq0$, $\rho\in \mathcal{M}_\gamma(\R^+\times \overline{O})$, $\rho\geq0$, and  {\normalfont(\ref{eqn:Eulerian_target})} holds and  {\normalfont(\ref{eqn:Skorokhod_evolution})} holds $\forall\ w \in C_{-\gamma}^{1,2}(\R^+\times \overline{O})$.
\end{definition}
Note that if $L\in C_{-\gamma}(\R^+\times \overline{O})$ any admissible pair has finite cost in \eqref{eqn:Eulerian_cost}.
	We recall the two formulations of the primal problem, from (\ref{eqn:Skorokhod_cost}) we have
	$$
		{\mathcal P}_0(\mu,\nu) :=  \inf
		\Big\{\mathbb{E}\Big[ \int_0^\tau L(t,B_t)dt\Big];\ \tau \in {\mathcal T}_O(\mu, \nu)
		\Big\},
	$$
	where $\tau_O(\mu,\nu)$ is the $\tau\in \mathcal{R}$, such that $B_t\in O$ for $t<\tau$ almost surely, with $B_0\sim \mu$ and $B_\tau\sim \nu$.  The second formulation  is given in (\ref{eqn:Eulerian_cost}) as the linear optimization problem over $(\eta,\rho)$ satisfying (\ref{eqn:Skorokhod_evolution}) and (\ref{eqn:Eulerian_target}),
	$$
	{\mathcal P}_1(\mu,\nu) = \inf_{\eta,{\rho}}\Big\{\int_{\overline{O}}\int_{\R^+} L(t,x)\eta(dt,dx);\ (\eta,{\rho})\ {\rm is\ admissible}\Big\}. 
	$$

 	We will eventually show that ${\mathcal P}_0(\mu,\nu)={\mathcal P}_1(\mu,\nu)$. In the following proposition, we show that one can associate an admissible pair to any randomized stopping time, which will mean that ${\mathcal P}_0(\mu,\nu)\geq {\mathcal P}_1(\mu,\nu)$.
	\begin{proposition}\label{thm:stochastic_embedding}
 		Given $\mu$ and $\nu$ satisfying {\normalfont \ref{itm:convex_support}},
		 if $\tau $ is a stopping time in ${\mathcal T}(\mu, \nu)$, then  	
	there is an admissible pair $(\eta,{\rho})$ such that for every $g\in C_{-\gamma}(\R^+\times \overline{O})$,
	$$
 			\E\big[g(\tau,B_\tau)\big]=\int_{\overline{O}}\int_{\R^+}g(t,x){\rho}(dt,dx),
 		$$
		and 
 	\begin{align}
 			\E\Big[\int_0^\tau g(t,B_t)dt\Big]   
 			=\ \int_{\overline{O}}\int_{\R^+}g(t,x)\eta(dt,dx).\nn
 		\end{align}	
		
	Furthermore, there is an $\eta$-measurable map $(t,x)\rightarrow \tau^{t,x}\in \mathcal{R}^{t,x}$ such that
	\begin{align}\label{eqn:conditional_expectation}
		\int_{\overline{O}}\int_{\R^+}\mathbb{E}^{t,x}\Big[A\big({\tau^{t,x}},B_{\tau^{t,x}}\big)\Big]\eta(dt,dx)= \int_{\R^+}\mathbb{E}\big[\mathbf{1}\{t\leq \tau\}A\big(\tau,B_\tau\big)\big]dt
	\end{align}
	for all $A_{-\gamma}\in C(\R^+;\R^d)$.
 	\end{proposition}
 	\begin{proof}
 		The expectation $\E[g(\tau,B_\tau)]$ defines a continuous linear functional on $C_{-\gamma}(\R^+\times \overline{O})$ thanks to the continuity of the Brownian paths.  
		Thus $\E[g(\tau,B_\tau)]$ is represented by ${\rho}\in \mathcal{M}_\gamma(\R^+\times \overline{O})$ such that for $g\in C_{-\gamma}(\R^+\times \overline{O})$,
 		$$
 			\E[g(\tau,B_\tau)]=\int_{\overline{O}}\int_{\R^+}g(t,x){\rho}(dt,dx).
 		$$
 		We immediately have that $\rho$ satisfies (\ref{eqn:Eulerian_target}) if and only if $B_\tau\sim \nu$.  Similarly, the density of the process $\eta$ is the representation of the linear functional
 		$$
 			\E\Big[\int_0^\tau g(t,B_t)dt\Big] =\int_{\overline{O}} \int_{\R^+} g(t,x)\eta(dt,dx).
 		$$
 		Ito's formula shows that if $w\in C_{-\gamma}^{1,2}(\R^+\times \overline{O})$ then
 		\begin{align}
 			\E\big[w(\tau,B_\tau)\big]  =&\ \E\Big[\int_0^\tau\Big(\frac{\partial}{\partial t}w(t,B_t)+ \frac{1}{2}\Delta w(t,B_t)\Big) dt\Big] + \E\big[w(0,B_0)\big]\nn\\
 			=&\ \int_{\overline{O}}\int_{\R^+}\Big(\frac{\partial}{\partial t}w(t,x)+ \frac{1}{2}\Delta w(t,x)\Big)\eta(dt,dx)\nn\\
 			&\ +\int_{\overline{O}}w(0,y)\mu(dy).\nn
 		\end{align}
 		Combining this expression with the representation of $\E[w(\tau,B_\tau)]$ by $\rho$ shows that $(\eta,{\rho})$ satisfy (\ref{eqn:Skorokhod_evolution}). 
		
		To prove (\ref{eqn:conditional_expectation}), given $\tau$ a randomized stopping time representing $\alpha:\Omega\rightarrow \mathcal{M}(\R^+)$, we define a measure $\mathbb{P}^\alpha$ on $\R^+\times \overline{O}\times \Omega\times \R^+$ by dual representation such that
		$$
			\int_{\R^+}\mathbb{E}\big[\mathbf{1}\{t\leq \tau\}H^{t,B_t}_\tau\big]dt = \int_{\R^+}\int_{\overline{O}}\int_\Omega \int_{\R^+}H^{t,x}_\tau(\omega)\mathbb{P}^\alpha(dt,dx,d\omega,d\tau)
		$$
		for all continuous functions $H$ on $\R^+\times \overline{O}\times \Omega \times \R^+$ with $\gamma$-exponential growth in the $t$ and $\tau$ variables. For $H^{t,x}_\tau(\omega)= f(t,x)$ we have, using the first part of the proposition,
		\begin{align*}
 			\int_{\overline{O}}\int_{\R^+}f(t,x)\eta(dt,dx)=&\ \E\Big[\int_0^\tau f(t,B_t)dt\Big]\\
 			 =&\ \int_{\R^+}\mathbb{E}\big[\mathbf{1}\{t\leq \tau\}f(t,B_t)\big]dt\\
 			=&\  \int_{\R^+}\int_{\overline{O}}\int_\Omega \int_{\R^+}f(t,x)\mathbb{P}^\alpha(dt,dx,d\omega,d\tau),
		\end{align*}
		and thus $\mathbb{P}^\alpha$ disintegrates as $\tilde{\alpha}^{t,x}(d\tau,d\omega)\eta(dt,dx)$.  Furthermore, for $H^{t,x}_\tau(\omega)=G^{t,x}(\omega)$, we have
		\begin{align*}
			\int_{\R^+}\int_{\overline{O}}\mathbb{E}^{t,x}\big[G^{t,x}\big]\eta(dt,dx) =&\ \int_{\R^+}\mathbb{E}\big[\mathbf{1}\{t\leq \tau\}G^{t,B_t}\big]dt\\
			=&\ \int_{\R^+}\int_{\overline{O}}\int_\Omega \int_{\R^+}G^{t,x}(\omega)\tilde{\alpha}^{t,x}(d\tau,d\omega)\eta(dt,dx),
		\end{align*}
		and we have that for $\eta$-a.e.\ $(t,x)$, the measure disintegrates further as 
		$$
			\mathbb{P}^\alpha(dt,dx,d\omega,d\tau)=\alpha^{t,x}(\omega)(d\tau)\mathbb{P}^{t,x}(d\omega)\eta(dt,dx).
		$$ 
		For $\eta$-a.e.\ $(t,x)$ and $\mathbb{P}^{t,x}$ a.e.\ $\omega$ we have that $\alpha^{t,x}(\omega)\geq 0$, $\int_t^\infty\alpha^{t,x}(\omega)(d\tau)=1$, and for $T\geq t$, $\omega \mapsto \int_t^T \alpha^{t,x}(\omega)(d\tau)$ is $\mathcal{F}_T^{t,x}$ measurable.  Thus we have that for $\eta$-a.e.\ $(t,x)$, $\alpha^{t,x}\in \mathcal{R}^{t,x}$, measurability follows from the disintegration, and formula (\ref{eqn:conditional_expectation}) follows with $H^{t,x}_\tau(\omega)=A(\tau,B_\tau(\omega))$ completing the proof.
 	\end{proof}
	
	\subsection{Eulerian dual $\mathcal{D}_1(\mu, \nu)$ and its weak duality}
We now make the dual problem $\mathcal{D}_1(\mu, \nu)$  more precise by considering 
\begin{align} \label{eqn:dual_cost.bis}
		{\mathcal D}_1(\mu,\nu) := \sup_{\psi,J}\Big\{\int_{\R^d}\psi(z)\nu(dz) - \int_{\R^d}J(0,y)\mu(dy);\ V_\psi[J]\geq 0\Big\},
\end{align}
where the end-potential, $\psi\in C(\overline{O})$, and the value-function, $J\in C_{-\gamma}^{1,2}(\R^+\times \overline{O})$. The maximization problem is posed over supersolutions $V_\psi[J]\geq 0$ for the Hamilton-Jacobi quasi-variational operator $V_\psi$ of (\ref{eqn:variational_operator}).

	In the remainder of this section we prove the weak duality, ${\mathcal P}_1(\mu,\nu)={\mathcal D}_1(\mu,\nu)$, by a method standard to the Monge-Kantorovich duality of optimal transportation.  We will then show that  ${\mathcal D}_1(\mu,\nu)={\mathcal D}_0(\mu,\nu)$ and conclude \eqref{eqn:dual_equalities} via the result  of  \cite{beiglboeck2017optimal} that shows ${\mathcal D}_0(\mu,\nu)={\mathcal P}_0(\mu,\nu)$.  
	 As the latter was proven  for a slightly different but equivalent problem in Theorem 4.2 \cite{beiglboeck2017optimal} (see also \cite{guo2016monotonicity} for a formulation similar to ours),  we shall 
	sketch  an alternative proof in Theorem \ref{thm:probabilistic_weak_duality}  in the appendix.

	\begin{theorem}\label{thm:Eulerian_weak_duality}
		Given {\normalfont\ref{itm:convex_support}} and {\normalfont\ref{itm:continuous}}, 
		 we have ${\mathcal D}_1(\mu,\nu)={\mathcal P}_1(\mu,\nu)$.
	\end{theorem}
	\begin{proof}
		This follows from the general Fenchel-Rockafellar duality \cite{V1} with $\Theta,\Xi:C_{-\gamma}(\R^+\times \overline{O})^2\rightarrow \R$ by  
\begin{align*}
 \Theta(f,g) := \begin{cases}
     0 & \text{if  $f\geq -L$ and $ g\geq 0$}, \\
    \infty  & \text{otherwise}.
\end{cases}
\end{align*}
	and, minimizing over $\psi\in C(\overline{O})$ and $J\in C_{-\gamma}^{1,2}(\R^+\times \overline{O})$,
		$$
			\Xi(f,g) := -\sup_{\psi,J}\Big\{\int_{\overline{O}} \psi(z)\nu(dz) - \int_{\overline{O}} J(0,y)\mu(dy);\ -\frac{\partial}{\partial t} J -\frac{1}{2}\Delta J\geq f,\ J-\psi\geq g\Big\}.
		$$
		Both $\Xi$ and $\Theta$ are convex and lower semicontinuous and by definition we have that ${\mathcal D}_1(\mu,\nu)=-\Xi(-L,0)$. 

		If there is $(f_1,g_1)$ such that $\Xi(f_1,g_1)=-\infty$, then $\Xi(f,g)=-\infty$ for all $f$ and $g$, and there cannot be an admissible $(\eta,\rho)$ with finite cost, thus ${\mathcal D}_1(\mu,\nu)={\mathcal P}_1(\mu,\nu)=+\infty$. 

		If $\Xi(f,g)$ is finite for all $f$ and $g$, then $\Theta$ is continuous at $f$ and $g$ if $f>-L$ and $g>0$ (identically $0$ in a neighborhood) and the Fenchel-Rockafellar duality theorem 
		(Theorem 1.9 in \cite{V1}) states that an optimizer $(\eta,\rho)$ exists and
		$$
			\max_{\eta,\rho}\big[ - \Theta^*(-\eta,-\rho)-\Xi^*(\eta,\rho)\big] = \inf_{f,g}\big[\Theta(f,g)+\Xi(f,g)\big].  
		$$
		The infimum on the right side is attained at $f=-L$ and $g=0$, in which case $\Theta(f,g)+\Xi(f,g)=-{\mathcal D}_1(\mu,\nu)$.  The only subtlety in checking continuity of $\Theta$ at $(f,g)$ is that for $f>-L$ we require that $\e{-\gamma t}f(t,x)\geq -\e{-\gamma t}L(t,x)+\delta$. However, this is not a problem because any admissible $(\eta,\rho)$ decays faster than $\e{\gamma t}$ and provides a bound showing that $\Xi(f,g)$ is finite.

		The Legendre transform of $\Theta$ is
		\begin{align}
			\Theta^*(-\eta,-\rho)=&\ \sup_{f,g}\Big\{ -\int_{\overline{O}}\int_{\R^+} f(t,x) \eta(dt,dx) -\int_{\overline{O}}\int_{\R^+} g(t,x) \rho(dt,dx)-\Theta(f,g)\Big\}\nn\\
			=&\ \begin{cases}
			\int_{\overline{O}}\int_{\R^+} L(t,x) \eta(dt,dx), & (\rho,\eta)\geq 0,\\
			\infty, & {\rm otherwise},
			\end{cases} \nn
		\end{align}
		whereas the Legendre transform $\Xi^*(\eta,\rho)$ is $0$ if (\ref{eqn:Eulerian_target}) and (\ref{eqn:Skorokhod_evolution}) hold and otherwise $+\infty$:
		\begin{align}
			 \Xi^*(\eta,\rho)& = \ \sup_{\psi,J} \Big\{\int_{\overline{O}} \int_{\R^+}\Big[-\frac{\partial}{\partial t}J(t,x)-\frac{1}{2}\Delta J(t,x)\Big]\eta(dt,dx)\nn \\ &+\int_{\overline{O}}\int_{\R^+} \big[J(t,x)-\psi(x)\big]\rho(dt,dx)
			+\int_{{O}}\psi(z)\nu(z)dz-\int_OJ(0,y)\mu(y)dy \Big\}.\nn
		\end{align}
		Thus ${\mathcal P}_1(\mu,\nu)=-\max_{\eta,\rho}\big[ - \Theta^*(-\eta,-\rho)-\Xi^*(\eta,\rho)\big]$ and the proof is complete.
	\end{proof}

	We now prove the identity ${\mathcal D}_1(\mu,\nu)={\mathcal D}_0(\mu,\nu)$ by using well known relationships between the Snell envelope and viscosity solutions \cite{el1997reflected}.  
	\begin{proposition}\label{prop:dual_equivalence}
		The dual problems are equivalent, i.e.,  ${\mathcal D}_1(\mu,\nu)={\mathcal D}_0(\mu,\nu)$.
	\end{proposition}
	\begin{proof} 
		One inequality is easy.	Given $\psi\in C(\overline{O})$ and $J\in C_{-\gamma}^{1,2}(\R^+\times \overline{O})$ satisfying $V_\psi[J]\geq 0$, we may consider the supermartingale $G^J\in \mathcal{K}_{-\gamma}^+$ defined by 
		\begin{align}
			G_t^J := J(t,B_t)-\int_0^t L(s,B_s)ds.\nn
		\end{align}
		That $G^J$ is supermartingale follows easily from $\frac{\partial}{\partial t}J(t,x)+\frac{1}{2}\Delta J(t,x)\leq L(t,x)$ and Ito's formula. Indeed, for $t\leq \tau\leq \tau_O$,
		\begin{align}
			\E\big[G_\tau^J| \mathcal{F}_t\big] =&\ \E\Big[\int_t^\tau\Big(\frac{\partial}{\partial t}J(t,B_t)+\frac{1}{2}\Delta J(t,B_t)-L(t,B_t)\Big)dt+\nabla J(t,B_t)\cdot dB_t| \mathcal{F}_t\Big]+G_t^J\nn\\
			\leq&\ G_t^J.\nn
		\end{align}
		With the above definition of $G^J$ we have that $G^J_t-\psi(B_t)\geq -\int_0^tL(s,B_s)ds$ and
		$
			\E[G_0^J]=\int_{\overline{O}}J(0,y)\mu(dy),
		$ 
		which clearly shows that ${\mathcal D}_1(\mu,\nu)\le {\mathcal D}_0(\mu,\nu)$. 

		For the reverse inequality (which also follows from Proposition \ref{thm:stochastic_embedding}), 
		we will show that  given $(\psi,G)\in C(\overline{O})\times \mathcal{K}_{-\gamma}^+$, then for any $\epsilon>0$, there is $J\in C_{-\gamma}^{1,2}(\R^+\times \overline{O})$ such that $V_\psi[J]\geq 0$ and
		$$
			-\int_{\overline{O}}J(0,y)\mu(dy)\geq -\mathbb{E}\big[G_0\big]-\epsilon.
		$$
	  	Since $J$ is admissible for the same $\psi$, then by taking the limit as $\epsilon\rightarrow 0$, we get that ${\mathcal D}_1(\mu,\nu)\geq {\mathcal D}_0(\mu,\mu)$. 

		We shall construct $J$ as  an approximation of the viscosity solution $J_\psi$ to $V_\psi[J_\psi]=0$, which may be defined for each $(t,x)$ as
		$$
			J_\psi(t,x) := \sup_{ \tau \in \mathcal{R}^{t,x}}\Big\{\mathbb{E}^{t,x}\Big[\psi(B_\tau)-\int_t^\tau L(s,B_s)ds\Big]\Big\}.
		$$
		 In particular, by compactness of the randomized stopping times and continuity of $\psi$, $L$ and the paths of Brownian motion, there is a randomized stopping time $\tau^{y}$ that attains this value at $(0,y)$ so that
		\begin{align}
			\int_{\overline{O}}J_\psi(0,y)\mu(dy) =&\ \int_{\overline{O}}\mathbb{E}^{0,y}\Big[\psi(B_{\tau^y})-\int_0^{\tau^y} L(t,B_t)dt\Big]\mu(dy)\nn\\
			=&\ \mathbb{E}\Big[\psi(B_{\tau})-\int_0^{\tau} L(t,B_t)dt\Big]\nn\\
			\leq&\ \mathbb{E}\big[G_{\tau}\big] \leq \mathbb{E}\big[G_0\big],\nn
		\end{align}
		where $\tau\in \mathcal{R}$ is defined as $\tau^{B_0}$ and the last inequality is implied by the supermartingale property of $G$. 
		Perron's method (see Lemma \ref{prop:Perron}) implies that $J_\psi$ may be approximated by a smooth supersolution, completing the proof.
	\end{proof}

\section{Attainment in Sobolev Class and Verification}\label{sec:attainment_verification}


\subsection{ Primal problem and regularity of optimal solutions}
	We define a couple more function spaces that will be useful in our analysis. We let $L^2_\gamma(\R^+;H_0^1(O))$ denote the weighted Sobolev space with norm
	\begin{align}\label{eqn:l2gamma}
		\|\phi\|_{L^2_\gamma(\R^+;H_0^1(O))}^2=\int_{\R^+}\e{2\gamma t}\int_{O}|\nabla \phi(t,x)|^2 dxdt,
	\end{align}
	which we will see is a natural space for the density of $\eta$. We also define a time dependent Sobolev space, $\mathcal{X}$, that is natural for the attainment of the value function $J$ for the dual problem $\mathcal{D}_1(\mu,\nu)$.  These are the functions $f$ that have $f\in L^2_{-\gamma}(\R^+;H_0^1(O))$ and a weak time derivative $\frac{\partial}{\partial t} f\in L^2_{-\gamma}(\R^+;H^{-1}(O))$, and the norm of $\mathcal{X}$ is
	\begin{align}\label{eqn:X}
		\|f\|_{\mathcal{X}}^2 = \int_{\R^+}\e{-2\gamma t}\Big[ \|f(t,\cdot)\|_{H_0^1(O)}^2+\|\frac{\partial}{\partial t}f(t,\cdot)\|_{H^{-1}(O)}^2\Big]dt.
	\end{align}
	Given $\psi\in H_0^1(O)$ lower semicontinuous, the unique `minimal weak solution' of $V_\psi[J_{\psi}]=0$ in $\mathcal{X}$ is shown to exist in \cite{bensoussan2011applications}; see Definition \ref{def:variational_inequality}.

		We note that the attainment of ${\mathcal P}_0(\mu,\nu)$ and ${\mathcal P}_1(\mu,\nu)$ in the weak class of $\mathcal{R}$ and $\mathcal{M}_{\gamma}(\R^+\times \overline{O})^2$ follows from the proof of Theorems \ref{thm:probabilistic_weak_duality} and \ref{thm:Eulerian_weak_duality}. We shall now prove that the evolution equation (\ref{eqn:Skorokhod_evolution}) implies additional regularity properties for $\eta$. We abuse notation to use $\eta$ to refer to the density with respect to the Lebesgue measure.

 	\begin{lemma}\label{lem:regularity}
 		Suppose {\normalfont \ref{itm:convex_support}} and {\normalfont \ref{itm:L2_initial}}, and that $(\eta,\rho)$ is admissible.  We let $\tilde{\mu}(t,x)$ be the solution to the heat equation, $\partial_t\tilde{\mu}=\frac{1}{2}\Delta \tilde{\mu}$, with the  Dirichlet boundary condition on $O$ and $\tilde{\mu}(0,\cdot)=\mu$.  Then $\eta\leq \tilde{\mu}$  a.e., moreover, there is a uniform bound on the norm of $\eta$ in the Sobolev space $L^2_\gamma(\R^+;H^1_0(O))$,
 		\begin{align}\label{eqn:uniform_bound}
 			\|\eta\|_{L^2_\gamma(\R^+;H^1_0(O))}^2\leq C\|\mu\|_{L^2(O)}^2,
 		\end{align}
 		and $\rho$ lies in the dual space of $\mathcal{X}$.
 	\end{lemma}
 	
 	\begin{proof}
 	The proof is standard but we give it here for completeness. We approximate $(\eta,\rho)$ by smooth functions $(\eta_\epsilon,\rho_\epsilon)$ by a parabolic mollification, e.g.\
 		$$
 			\eta_\epsilon(t,x)= \int_{|y-x|<\epsilon}\int_t^{t+\epsilon^2} n_\epsilon(s-t,y-x)\eta(ds,dy)
 		$$
 		where 
 		$$
 			n_\epsilon(t,x):= \frac{1}{\epsilon^{d+2}}n\left(\frac{t}{\epsilon^2},\frac{x}{\epsilon}\right)
 		$$
 		for $n$ a smooth nonnegative function supported in $(0,1)\times B(1,0)$ that integrates to 1.

 		From the equation \eqref{eqn:Skorokhod_evolution} we have that 
		\begin{align}\label{eqn:smooth_evolution}
			\rho_\epsilon(t,x)+\frac{\partial}{\partial t} \eta_\epsilon (t,x)-\frac{1}{2}\Delta \eta_\epsilon(t,x)=0
		\end{align}
		with $\eta_\epsilon,\rho_\epsilon\geq 0$, and $\mu_\epsilon=\eta_\epsilon(0,\cdot)+\rho(\{0\},\cdot)\rightarrow \mu$ in $L^2(O)$ as $\epsilon \to 0$. Moreover, notice that the equation \eqref{eqn:Skorokhod_evolution} and the condition \ref{itm:convex_support} implies that $\eta_\epsilon$ satisfies the Dirichlet boundary condition for $O$, for sufficiently small $\epsilon$.

		We assume $\epsilon$ is sufficiently small that $\mu_\epsilon$ has support in $O$, and we define $\tilde{\mu}_\epsilon$ to be the solution of the heat equation with Dirichlet boundary conditions on $\partial O$ and $\tilde{\mu}_\epsilon(0,x)=\mu_\epsilon(x)$,

 		We let $u(t,x)=\eta_\epsilon(t,x)-\tilde{\mu}_\epsilon(t,x)$, so that $u(0,x)=0$, $u= 0$ in $\partial O$,  and
		$$
			\frac{\partial}{\partial t} u(t,x)\leq \frac{1}{2}\Delta u(t,x).
		$$
		The parabolic maximum principle implies that $u\leq 0$ and hence $\eta_\epsilon \leq \tilde{\mu}_\epsilon$, and by taking $\epsilon \to  0$, we get $\eta \le \tilde{\mu}$ a.e.

 		We use $\e{2\gamma t}\eta_\epsilon$ as a test function for (\ref{eqn:smooth_evolution}) to obtain
		\begin{align}
			0=&\ \int_{\R^+}\e{2\gamma t}\int_O\Big[ \rho_\epsilon(t,x)\eta_\epsilon(t,x)+\frac{1}{2}|\nabla \eta_\epsilon(t,x)|^2\Big]dxdt\nn\\
			&\ +\int_{\R^+}\int_O\Big[\frac{1}{2}\frac{\partial}{\partial t} \e{2\gamma t}\eta_\epsilon(t,x)^2-\gamma \e{2\gamma t}\eta_\epsilon(t,x)^2\Big]dxdt\nn\\
			=&\ \int_{\R^+}\e{2\gamma t}\int_O\Big[\rho_\epsilon(t,x)\eta_\epsilon(t,x)+\frac{1}{2}|\nabla \eta_\epsilon(t,x)|^2\Big]dxdt \nn\\
			&\ -\frac{1}{2}\|\eta_\epsilon(0,\cdot)\|_{L^2(O)}^2-\int_{\R^+} \gamma \e{2\gamma t}\|\eta_\epsilon(t,\cdot)\|_{L^2(O)}^2dt.\nn
		\end{align}
		The first term (with $\rho_\epsilon \eta_\epsilon$) is nonnegative and can be dropped.  The last two terms are bounded by comparison with the solution $\tilde{\mu}_\epsilon$.  For $\lambda$ the smallest eigenvalue of $-\frac{1}{2}\Delta$ with Dirichlet boundary conditions, 
		$$
			\|\tilde{\mu}_\epsilon(t,\cdot)\|_{L^2(O)}^2\leq \e{-2\lambda t}\|\mu_\epsilon\|_{L^2(O)}^2.
		$$
		Since we have that $\|\mu_\epsilon\|_{L^2(O)}\rightarrow \|\mu\|_{L^2(O)}$ as $\epsilon\rightarrow 0$, the estimate (\ref{eqn:uniform_bound}) follows for $\gamma<\lambda$.  With the additional regularity of $\eta$, and the embedding of $\mathcal{X}\subset C_{-\gamma}(\R^+;L^2(O))$, we have that $\rho$ is a continuous linear functional of $\mathcal{X}$.
 	\end{proof}

We can now deduce the following. 
	\begin{proposition}\label{prop:Skorokhod_existence}
		We suppose {\normalfont \ref{itm:convex_support}}, {\normalfont \ref{itm:L2_initial}} and {\normalfont\ref{itm:continuous}}.  Then, ${\mathcal P}_1(\mu,\nu)$ is finite 
		 if and only if {\normalfont \ref{itm:subharmonic_order}} holds. Moreover, there exists an optimal nonnegative admissible pair $(\eta,\rho)\in L^2_\gamma(\R^+;H_0^1(O))\times \left(\mathcal{X}^*\cap\mathcal{M}_\gamma (\R^+\times \overline{O})\right)$ for ${\mathcal P}_1(\mu,\nu)$. 
	\end{proposition}
	\begin{proof}
	First, due to weak* compactness of the space $\mathcal{M}_\gamma(\R^+\times\overline{O})^2$ as well as the duality  	 
	Theorem \ref{thm:Eulerian_weak_duality},
	there exists optimal  admissible $(\eta,\rho)\in \mathcal{M}_\gamma(\R^+\times\overline{O})^2$ with finite cost if and only if $\mathcal{P}_1(\mu, \nu) = {\mathcal D}_1(\mu,\nu)<\infty$.  If \ref{itm:subharmonic_order} fails, then there is $h\in C^2(\overline{O})$ with $-\frac{1}{2}\Delta h\geq 0$ and
		$$
			\int_{\overline{O}} h(z) \nu(dz)>\int_{\overline{O}} h(y) \mu(dy).
		$$
		We let $\psi=J=\alpha h$ and $J$ satisfies $V_\psi[J]\geq 0$ and
		$$
			\int_{\overline{O}} \psi(z) \nu(dz)-\int_{\overline{O}} J(0,y)\mu(dy)\rightarrow +\infty
		$$
		as $\alpha\rightarrow +\infty$, making $\mathcal{D}_1(\mu, \nu) =+\infty$.

		Given \ref{itm:subharmonic_order} 
		 we know that $\mathcal{T}(\mu,\nu)\ne \emptyset$, and Proposition \ref{thm:stochastic_embedding} implies there exists an admissible pair $(\eta,\rho)$ thus ${\mathcal D}_1(\mu,\nu)$ is finite.  The proof of the proposition is concluded by applying Lemma \ref{lem:regularity} to show that if $(\eta,\rho)$ exists then $(\eta,\rho)\in L^2_\gamma(\R^+;H_0^1(O))\times \left(\mathcal{X}^*\cap\mathcal{M}_\gamma (\R^+\times \overline{O})\right)$.
	\end{proof}
\subsection{ Dual attainment in Sobolev class}
	We now address the attainment of the dual problem ${\mathcal D}_1(\mu, \nu)$.
		The upcoming theorem stating that the dual problem ${\mathcal D}_1(\mu, \nu)$ is attained at $(\psi,J)\in H_0^1(O)\times \mathcal{X}$ is one of our most important results.  We note that $V_\psi[J]\geq 0$ and $V_\psi[J]=0$ both make sense as a weak variational inequality in these spaces; see Definition \ref{def:variational_inequality}. Furthermore, due to the embedding of $\mathcal{X}\subset C_{-\gamma}(\R^+;L^2(O))$ 
	 the dual value of $\psi$ and $J$,  that is,
	\begin{align}\label{eqn:dual_value}
		\int_{\overline{O}} \psi(z)\nu(dz)-\int_{\overline{O}} J(0,y)\mu(dy),
	\end{align}
	is well defined given \ref{itm:L2_initial}. We note that there are three ways that $V_\psi[J]\geq 0$ inadequately bounds $\psi$ and $J$.  
	\begin{itemize}
		\item Subtracting a positive function from $\psi$ clearly decreases the value (\ref{eqn:dual_value}). This is addressed easily, for example by supposing $\psi(x)=\inf\{t;\ J(t,x)\}$, although we follow a slightly different approach.
		\item Adding a non-negative supersolution, $-\partial_t w - \frac{1}{2}\Delta w\geq 0$ to $J$, also decreases the value (\ref{eqn:dual_value}), and is addressed easily by supposing $V_\psi[J]=0$.
		\item The most sensitive degree of freedom comes from adding 
		a superharmonic function to both $\psi$ and $J$, which also decreases the value (\ref{eqn:dual_value}) given \ref{itm:subharmonic_order}.
	\end{itemize}

	We now prove a result that utilizes this intuition to provide the needed estimates for the proof of dual attainment.  In particular, we restrict to a class of solutions such that the weak notion of $V_\psi[J]=0$ coincides with the viscosity sense.

	\begin{proposition}\label{prop:dual_estimates}
		We suppose {\normalfont\ref{itm:convex_support}},{\normalfont\ref{itm:L2_initial}}, {\normalfont\ref{itm:subharmonic_order}},  {\normalfont   \ref{itm:continuous}} and {\normalfont \ref{itm:bounded}}.  Maximizing {\normalfont (\ref{eqn:dual_value})} over the set  of $\psi$ and $J$ satisfying the following has the same value as $\mathcal{D}_1(\mu,\nu)$ in \eqref{eqn:dual_cost}:
		\begin{enumerate}[label =\bf \roman*)]
			\item\label{i} $\psi\in H_0^1(O)$, $J\in \mathcal{X}$, and $J$ satisfies $V_\psi[J]\geq 0$ weakly;
			\item\label{ii} $\psi(x)\leq J(t,x)\leq 0$\ for almost all $(t,x)\in \R^+\times \overline{O}$;
			\item\label{iii} $-\frac{1}{2}\Delta \psi(x)\geq -D:=\|L\|_\infty$ weakly.
		\end{enumerate}	
	\end{proposition}

	\begin{proof}
		When  maximizing (\ref{eqn:dual_value}) over $\psi\in H_0^1(O)$ and $J\in \mathcal{X}$ satisfying $V_\psi[J]\geq 0$ weakly, the supremum is bounded above by $\mathcal{P}_1(\mu,\nu)=\mathcal{D}_1(\mu,\nu)$ because $\psi$ and $J$ have sufficient regularity to be used as test functions in (\ref{eqn:Eulerian_target}) and (\ref{eqn:Skorokhod_evolution}), which implies the weak duality inequality. 

		If $\psi\in C(\overline{O})\cap H_0^1(O)$ then the viscosity solution $J_\psi$ is in $C(\R^+\times \overline{O})\cap \mathcal{X}$ and is the unique weak solution to $V_\psi[J]=0$, which is also the smallest supersolution (infimum over supersolutions); see Proposition \ref{prop:viscosity_variational_equivalence}. Due to this, in the following three stage process, weak solutions and viscosity solutions may be used interchangeably.

	{\bf Step 1:}	Given $\psi_0\in C(\overline{O})$ and $J_0\in C_{-\gamma}^{1,2}(\R^+\times \overline{O})$  with $V_{\psi_0} [J_0] \ge 0$, we can first approximate $\psi_0$ from below by $\hat{\psi}\in C^\infty(\overline{O})$ such that $\psi_0-\epsilon \leq \hat{\psi}\leq \psi_0$, and therefore
		$$
			\int_{\overline{O}}\hat{\psi}(z)\nu(dz)\geq \int_{\overline{O}}{\psi_0}(z)\nu(dz)-\epsilon \quad \&  \quad V_{\hat \psi} [J_0]\ge V_{\psi_0} [J_0] \ge 0. 
		$$
		Then we define $h$ to be the solution to the Poisson problem $\Delta h =0$ with $h=\hat{\psi}$ on $\partial O$ and we let $\psi_1=\hat{\psi}-h$.  For $\hat{J}=J_0-h$ we have  $V_{\psi_1}[\hat{J}]\geq 0$. Given \ref{itm:subharmonic_order}, (\ref{eqn:dual_value}) does not change upon replacing $(\hat{\psi},J)$ with $(\psi_1,\hat{J})$.   We let $J_1$ be the unique solution to $V_{\psi_1}[J_1]=0$ with Dirichlet boundary condition, and $J_{1}\in C_{-\gamma}(\R^+\times \overline{O})\cap \mathcal{X}$.  Then $(\psi_1,J_1)$ satisfy item \ref{i}, 
		and as $J_1$ is the smallest supersolution, we immediately have $J_{1}\leq \hat{J}$, hence (\ref{eqn:dual_value}) again does not decrease.

		{\bf Step 2:} We next define $\phi\in C(\overline{O})\cap H_0^1(O)$ as the superharmonic envelope of $\psi_1$, equivalently 
		the viscosity solution given by the infimum over supersolutions of
		$$
			\min\left\{\phi(x)-\psi_1(x),-\frac{1}{2}\Delta\phi(x)\right\}=0.
		$$
	We define $\psi_2:=\psi_1(x)-\phi(x)$ and $J_2(t,x):=J_1(t,x)-\phi(x)$. Clearly, $J_2\geq \psi_2$. Since $\phi$ is a supersolution of $V_{\psi_1}[\phi]\geq 0$ we have  
	$\phi(x)\geq J_1(t,x)$ and thus $J_2\leq 0.$ If $J_2(t_1,x_1)<0$ then $\phi(x_1)>J_1(t_1,x_1)\geq \psi_1(x_1)$ and 
		$$
			-\frac{\partial}{\partial t} J_2(t_1,x_1)-\frac{1}{2}\Delta J_2(t_1,x_1)= -\frac{\partial}{\partial t} J_1(t_1,x_1)-\frac{1}{2}\Delta J_1(t_1,x_1)\geq -L(t_1,x_1).
		$$
		If $J_2(t_1,x_1)=0$, then since $J_2\leq 0$ it follows that
		$$
			-\frac{\partial}{\partial t} J_2(t_1,x_1)-\frac{1}{2}\Delta J_2(t_1,x_1)\geq 0\geq -L(t,x).
		$$
		We have shown that $J_2$ is a supersolution  
		$$V_{\psi_2}[J_2]\geq 0,$$ and $(\psi_2,J_2)$ satisfy item 
		\ref{ii}.  
  		Moreover notice that 
		the difference in dual value is 
		$$
			\int_O [-\phi(z)]\nu(z)dz-\int_O [-\phi(y)]\mu(y)dy\geq 0
		$$
		because $-\phi$ is subharmonic and $\mu\prec\nu$ by \ref{itm:subharmonic_order}.

{\bf Step 3:}
	We now modify the functions $(\psi_2, J_2)$ to satisfy all the properties  \ref{i}, \ref{ii}, and \ref{iii}. 
	We consider $\psi_3\in C(\overline{O})\cap H_0^1(O)$ that solves the variational inequality
		\begin{align}\label{eqn:var-D}
		\min\Big\{\psi_3-\psi_2,-\frac{1}{2}\Delta \psi_3+D\Big\} = 0.
		\end{align}
		Here, we can take $\psi_3$ so that it is the smallest supersolution.
		Clearly, $\psi_3\geq \psi_2$ and $\psi_3$ satisfies the item \ref{iii}.

Consider
$\displaystyle \tilde J_2(x) : = \inf_t J_2 (t,x)$.  Notice that for each $x$ either the infimum is attained at some finite $t$ or  there is a sequence $t_i \to \infty$ such that $\displaystyle \lim_{i \to \infty} 
\frac{\partial}{\partial t} J_2 (t_i, x) = 0$ and $\displaystyle \lim_{i\to \infty} J(t_i, x) =\tilde J (x)$. As $-\partial_t J_2 - \frac{1}{2} \Delta J_2  \ge -L \ge -D$, this implies that $\tilde J$ is a supersolution to \eqref{eqn:var-D}, therefore
\begin{align*}
 \psi_3 \le \tilde J_2 \le J_2. 
\end{align*}

		Notice that 
		$J_2$ is also a supersolution of $V_{\psi_3}[J_2]\geq 0$. 	Therefore, if we let $J_3$ be the smallest supersolution  to solve $V_{\psi_3}[J_3]=0$ then  finally $(\psi_3,J_3)$ satisfy items \ref{i}, \ref{ii}, \ref{iii}, 
		and  as $\psi_3 \ge \psi_2$ and $J_3 \le J_2$, we did not decrease the dual value more than $\epsilon$. This completes the proof.
	\end{proof}

	\begin{remark}
		Given $\psi$ satisfying items \ref{i}, \ref{ii}, \ref{iii},
		we may always select a representative that is lower semicontinuous (denoted by $LSC(\overline{O}))$.  Then the value function $J_\psi$ is lower semicontinuous and is a `discontinuous viscosity solution' to $V_\psi[J]=0$ \cite{ceci2004mixed}.  Furthermore it is verified in \cite{bensoussan2011applications} (Chapter 3, Section 2) that the value function to the optimal stopping problem with terminal reward $\psi$ is the ``minimal weak solution'' to $V_\psi[J_\psi]=0$.  However, $J_\psi$ no longer coincides with the infimum over smooth supersolutions to $V_\psi[J]\geq 0$, e.g.\ consider a case where $J_\psi\equiv\psi$ and $\psi$ is not upper semicontinuous.
	\end{remark}

	The proof of dual attainment is now straightforward using Proposition \ref{prop:dual_estimates}.
	\begin{theorem}\label{thm:strong_duality}
		Suppose {\normalfont \ref{itm:convex_support}, {\normalfont \ref{itm:L2_initial}}, \ref{itm:subharmonic_order}}, {\normalfont  \ref{itm:continuous}} and {\normalfont \ref{itm:bounded}}. The maximum of the Eulerian dual problem is attained at $\psi\in  LSC(\overline{O})\cap H_0^1(O)$, and $J_\psi\in  LSC(\overline{O})\cap\mathcal{X}$ which is the minimal weak solution of $V_\psi[J_\psi]=0$.
		
		For $(\eta,\rho)$ that minimize the Eulerian problem $\mathcal{P}_1(\mu,\nu)$ the complementary slackness condition holds:
		\begin{align}\label{eqn:complementary_slackness_S}
 	 		0=&\ \int_{{O}}\int_{\R^+}\Big[J(t,x)-\psi(x)\Big]{\rho}(dt,dx)\nn\\
 	 		&\ + \int_{\R^+}\int_{O}\Big[-\eta(t,x)\frac{\partial}{\partial t}J(t,x) +\frac{1}{2} \nabla J(t,x)\cdot\nabla \eta(t,x)+L(t,x)\eta(t,x)\Big] dxdt.
 	 	\end{align}

	\end{theorem}
	\begin{proof}
		We consider a maximizing sequence to the dual problem.
		By Proposition 
		\ref{prop:dual_estimates} 
		we may  assume that 
		$(\psi^i,J^i)\in H_0^1(O)\times \mathcal{X}$ with $V_{\psi^i}[J^i]\geq 0$,  
		$\psi^i\leq J^i\leq 0$, 
		and $-\Delta \psi^i(x)\geq -D$.  
		The last two conditions and the maximum principle imply that there is a constant $K$ depending on $D$ and $O$ such that $\psi^i\geq -K$. Therefore, we have the uniform bound on $\|\psi^i\|_{H_0^1(O)}$ given by
		\begin{align}
			\int_O |\nabla \psi^i(x)|^2dx =  &\  \int_{{O}} -\Delta \psi^i(x) \psi^i(x) dx\nn\\  
			\leq &\ \int_{{O}} D |\psi^i(x)|dx\leq DK|O|.
		\end{align}
		Then there is a subsequence such that ${\psi}^{i_k}\rightharpoonup \psi$ in $H_0^1(O)$. 
		Similarly, $J^{i_k}\rightharpoonup J$ in $\mathcal{X}$ (since $\|J^{i_k}\|_{\mathcal{X}}\leq C\|\psi^{i_k}\|_{H_0^1}$ by Proposition \ref{prop:viscosity_variational_equivalence}) with $V_\psi[J]\geq 0$ and $\psi$ and $J$ attain the maximum dual value.    Properties \textbf{ii.}\ and \textbf{iii.}\ are preserved in the weak limit, and Proposition \ref{prop:lower_semicontinuity} shows that $\psi$ has a lower semicontinuous representative.

		The optimizers $J$ and $\psi$ 
		now have  sufficient regularity to be used as   test functions in (\ref{eqn:Skorokhod_evolution}) and (\ref{eqn:Eulerian_target}) 
		to obtain
		\begin{align}\label{eqn:J_test}
 	 		\int_{{O}}J(0,y)\mu(y)dy=&\ \int_{\R^+}\int_{O}\Big[-\eta(t,x)\frac{\partial}{\partial t}J(t,x)+\frac{1}{2}\nabla J (t,x)\cdot\nabla\eta(t,x)\Big]dxdt \nn\\
		 	&\ +\int_{{O}}\int_{\R^+} J(t,x){\rho}(dt,dx); \\
		 	\int_{\overline{O}} \psi(z)\, \nu(dz) =&\ \int_{\overline{O}}\int_{\R^+} \psi (x){\rho}(dt,dx)\nn
		\end{align}
		In view of Theorem \ref{thm:Eulerian_weak_duality} we have
		$$
			\int_O \psi(z)\nu(z)dz-\int_{{O}}J(0,y)\mu(y)dy=\int_{\R^+}\int_O L(t,x)\eta(t,x)dxdt
		$$
		and add to (\ref{eqn:J_test}) we get (\ref{eqn:complementary_slackness_S}).
		Finally, we may also replace $J$ by the minimal weak solution to $V_\psi[J_\psi]=0$. 
	\end{proof}

	As a consequence of Theorem \ref{thm:strong_duality} we also have attainment of the problem $\mathcal{D}_0(\mu,\nu)$ in a weaker regularity class.

	\begin{corollary}\label{cor:martingale_attainment}
		Suppose the same hypotheses as Theorem {\normalfont\ref{thm:strong_duality}}, then the optimizers $(\psi,J_\psi)$ may be selected to satisfy the dynamic programming principle
		\begin{align}\label{eqn:dynamic_programming_2}
			J_\psi(t,x):=\sup_{ \tau \in \mathcal{R}^{t,x}}\Big\{\mathbb{E}^{t,x}\Big[\psi(B_\tau)-\int_t^\tau L(s,B_s)ds\Big]\Big\}.
		\end{align}
		Moreover, the process
		$$
			G_t:=J_\psi(t,B_t)-\int_0^t L(s,B_s)ds
		$$
		is a lower semicontinuous supermartingale (assuming $G_t=G_{t\wedge \tau_O}$ for the exit time $\tau_O$), satisfies $G_t-\psi(B_t)\geq -\int_0^tL(s,B_s)ds$, and attains the optimal cost
		$$
			\int_{\overline{O}} \psi(z)\nu(dz)-\mathbb{E}\big[G_0\big]=\mathcal{P}_0(\mu,\nu).
		$$
		For $\tau$ that minimizes the primal problem $\mathcal{P}_0(\mu,\nu)$, we have that $G_{t\wedge \tau}$ is a martingale and $G_\tau = \psi(B_\tau)-\int_0^\tau L(t,B_t)dt$.
	\end{corollary}
	\begin{proof}
		Given lower semicontinuous $\psi$, the minimal weak solution $J$ to $V_\psi[J]=0$ (with the Dirichlet boundary condition) coincides with the value function $J_\psi$ that satisfies \eqref{eqn:dynamic_programming_2}; see Theorem 4.6 of Chapter 3 in \cite{bensoussan2011applications} and note that since $\psi$ is lower semicontinuous it may be approximated from below by continuous functions.

		By Proposition \ref{prop:dual_equivalence}, $(\psi,G)$ also attain the value of $\mathcal{D}_0(\mu,\nu)$, although $G$ is only lower semicontinuous on the space of paths. The process $G$ satisfies 
		$$G_t-\psi(B_t)\geq -\int_0^tL(s,B_s)ds,$$ 
		for $t<\tau_O$ and $(\psi,G)$ has the same value that equals $\mathcal{P}_0(\mu,\nu)$.  That $G$ is a supermartingale follows from the dynamic programming principle (\ref{eqn:dynamic_programming_2}):
		$$
			G_t \geq \mathbb{E}\Big[ J_\psi(\sigma,B_\sigma)-\int_t^\sigma L(s,B_s)ds\Big|\mathcal{F}_t\Big]-\int_0^tL(s,B_s)ds=\mathbb{E} \big[G_\sigma\big|\mathcal{F}_t\big]
		$$  
		for any stopping time $t\leq\sigma\leq \tau_O$.

		If $\tau$ minimizes $\mathcal{P}_0(\mu,\nu)$ then
		$$
			\mathbb{E}\Big[\int_0^\tau L(t,B_t)dt\Big] = \mathbb{E}\big[\psi(B_\tau)-J_\psi(0,B_0)\big],
		$$
		which implies that $\mathbb{E}[G_\tau]\geq \mathbb{E}[G_0]$.  For any stopping time $0\leq \sigma \leq \tau$, since $G$ is a supermartingale $\mathbb{E}[G_\tau]\leq \mathbb{E}[G_\sigma]\leq \mathbb{E}[G_0]$, and it follows that $G_{t\wedge \tau}$ is a martingale and $G_\tau = \psi(B_\tau)-\int_0^\tau L(t,B_t)dt$ almost surely.
	\end{proof}

\section{Optimal Stopping by Hitting of Barrier}\label{sec:hitting_times}

 	We wish to characterize when the optimal stopping time is given by the hitting time of a barrier. 
	From Theorem \ref{thm:strong_duality} the process must stop in the set where $J(t,x)=\psi(x)$. We find that this set is given by the epi/hypergraph of a function $s(x)$ if $L$ is either strictly increasing or decreasing in time. In these cases the hitting time must be the first time the process enters this set. We denote the cases of time dependence of $L$ as:

 	\begin{enumerate}[label=(\textbf{D\arabic*})]
 		\item \label{itm:compounded} The cost function $L$ strictly increases in time;
	 	\item \label{itm:discounted} The cost function $L$ strictly decreases in time;
	 	\item \label{itm:stationary} The cost function $L$ is stationary in time.
	\end{enumerate}

\begin{remark}
 One may compare the results in this section to those  by Beiglb\"{o}ck, Cox and Huesmann \cite[Section 7]{beiglboeck2017optimal}, where they obtain similar results in cases comparable to those we consider below. While their approach is based on a stochastic variational principle on path space analogous to cyclic monotonicity in optimal transport, ours uses the Eulerian formulation and the corresponding dual variational inequalities  developed in the previous sections insipired by the Kantorovich duality approach to optimal transport. 
\end{remark}

	We start with the following monotonicity and regularity results for the dual functions. 
	\begin{proposition}\label{prop:dual_monotonicity}
		Given {\normalfont\ref{itm:continuous}}, we suppose that $\psi\in LSC(\overline{O})\cap H_0^1(O)$ and $J_\psi$ is the minimal weak solution of $V_\psi[J_\psi]=0$.

		\begin{itemize}
			\item Given {\normalfont\ref{itm:compounded}} then $J_\psi$ is non-increasing in time.  
			\item Given {\normalfont\ref{itm:discounted}} then $J_\psi$ is non-decreasing in time. 
			\item Given {\normalfont\ref{itm:stationary}} then $J_\psi$ is stationary in time.
		\end{itemize}
		In each case we define the free boundary
		\begin{align}\label{eqn:free_boundary}
			s(x):=\begin{cases}
				\inf\{t\in \R^+;\ J_\psi(t,x)=\psi(x)  \} & {\normalfont\ref{itm:compounded}} \\
				\sup\{t \in \R^+;\ J_\psi(t,x)=\psi(x) \} & {\normalfont\ref{itm:discounted}}\\
				 0  & {\normalfont\ref{itm:stationary}}.
		\end{cases}
		\end{align}
	\end{proposition}

	\begin{proof}
		We use the dynamic programming principle for the value function (\ref{eqn:dynamic_programming_2}) for the proof.  We first suppose {\normalfont\ref{itm:compounded}} and fix $x$, $t_1\leq t_2$.  For any $\epsilon>0$, there is a stopping time $\tau_2\in \mathcal{R}^{t_2,x}$, i.e.\ $\tau_2\geq t_2$, that nearly achieves the supremum of (\ref{eqn:dynamic_programming_2}) such that
		$$
			J_\psi(t_2,x) \leq \mathbb{E}^{t_2,x}\Big[ \psi\big(B_{\tau_2}\big)-\int_{t_2}^{\tau_2} L\big(s,B_s\big)ds\Big] +\epsilon.
		$$
		Then we let $\tau_1 = \tau_2+t_1-t_2 \in \mathcal{R}^{t_1,x}$ and the dynamic programming principle implies
		$$
			J_\psi(t_1,x) \geq \mathbb{E}^{t_1,x}\Big[ \psi\big(B_{\tau_1}\big)-\int_{t_1}^{\tau_1} L\big(s,B_s\big)ds\Big]\geq J_\psi(t_2,x)-\epsilon.
		$$
		For the second inequality we use {\normalfont\ref{itm:compounded}} as $-L(s,y)\geq -L(s+t_2-t_1,y)$. Taking $\epsilon$ to 0 shows that $J_\psi(t_1,x) \geq J_\psi(t_2,x)$.

		Given {\normalfont\ref{itm:discounted}} the proof follows the same line by first selecting $\tau_1$ to be nearly optimal starting from $t_1$ and defining $\tau_2=\tau_1+t_2-t_1$.  The same argument yields the inequality $J_\psi(t_2,x)\geq J_\psi(t_1,x)$.   

		With {\normalfont\ref{itm:stationary}}, both arguments are valid, showing $J_\psi(t_1,x)=J_\psi(t_2,x)$.
	\end{proof}

	\begin{remark}
		An alternate proof of monotonicity for $J_\psi$ by viscosity solution methods is possible if $\psi$ is continuous. For example given {\normalfont\ref{itm:compounded}}, if we consider
		$$
			\hat{J}(t,x) := \sup\big\{J_\psi(r,x);\ t\leq r\},
		$$
		obviously $\hat{J}\geq J_\psi$ and $\hat{J}$ is non-increasing in time.  Then we can show that $\hat{J}$ is a viscosity subsolution to $V_\psi[\hat{J}]\leq 0$ by translating the comparison function, after which the comparison principle proves $\hat{J}=J_\psi$.
	\end{remark}

	The  monotonicity has allowed us to define the free boundary 
	$s(x)$ of (\ref{eqn:free_boundary}). Then for the case of \ref{itm:compounded} (resp.\ \ref{itm:discounted}) we have the barrier set 
	$$R_c:= \{(t,x);\ t\geq (\hbox{resp.} \le)\, s(x)\}$$ and its `interior' $$R_o:=\{(t,x);\ t> (\hbox{resp.} <)\, s(x)\}.$$

	  We let $\tau_o$ and $\tau_c$ denote their corresponding first hitting-time: 
	\begin{align}\label{eqn:hitting_time}
		\tau_o= \inf\{t; B_t\in R_o\},\ \ \ \tau_c= \inf\{t; B_t\in R_c\}.
	\end{align}

	\begin{remark}The cases {\normalfont \ref{itm:compounded}} and {\normalfont\ref{itm:discounted}} correspond to the classical Root and Rost embeddings, respectively (see Figure {\normalfont(\ref{fig:inc_dec})}).  The case that $L$ decreases until some fixed time $t_0$ and later increases also implies the corresponding monotonicity of $J_\psi$ changing accordingly.  These so-called `cave embeddings' have been pointed out in \cite{beiglboeck2017optimal}. A later work by Cox and Kinsley \cite{cox2017robust} deal with a specific one-dimensional non-Lagrangian cost that switches between Root and Rost-like frameworks along a fixed curve $t_0(x)$. They then identify an explicit `cave barrier' that determines the embedding by using a PDE heuristic, which would be interesting to see if it could be made rigorous by a suitable refinement of our analysis. 
\end{remark}

	The following Lemma shows how the barriers determine the density and stopping distribution of the process given that the process always stops in $R_o$ and continues in the complement of $R_c$.  In the case of \ref{itm:compounded} there is only one such admissible pair ($\eta,\rho$), whereas for \ref{itm:discounted} the uniqueness holds up to the choice of how much mass to stop at time $0$ on the set where $s(x)=0$.

 	\begin{lemma}\label{lem:Eulerian_uniqueness}
	Suppose {\normalfont \ref{itm:convex_support}, {\normalfont \ref{itm:L2_initial}}, \ref{itm:subharmonic_order}}, {\normalfont   \ref{itm:continuous}} and {\normalfont \ref{itm:bounded}}. 
		We suppose $R$ is a measurable forward-barrier such that $(r,x)\in R$ whenever $(t,x)\in R$ with $t\leq r$, and $(t,x)\in R$ if there is $(t_i,x)\in R$ with $t_i\rightarrow t$.  Then given $\mu$ there is a unique solution $(\eta,\rho)$ to {\normalfont(\ref{eqn:Skorokhod_evolution})} with $\eta(R)=0$ and $\rho(R)=1$.

		If instead $R$ is a measurable backward-barrier such that $(r,x)\in R$ whenever $(t,x)\in R$ with $t\geq r$, and $(t,x)\in R$ if there is $(t_i,x)\in R$ with $t_i\rightarrow t$, then the solution $(\eta,\rho)$ to {\normalfont(\ref{eqn:Skorokhod_evolution})} with $\eta(R)=0$ and $\rho(R)=1$ is uniquely determined given the value of $\rho$ on the set $(0,x)$ where $s(x)=0$.
	\end{lemma}

	\begin{proof}
		We consider $(\sigma,\pi)\in L_{\gamma}^2(\R^+;H_0^1(O))\times \mathcal{M}_\gamma\cap \mathcal{X}^*$ that satisfy (\ref{eqn:Skorokhod_evolution}) with source distribution $\alpha\in L^2(O)$ and $\pi(\{0\},\cdot)\in L^2(O)$,
		\begin{align*}
		-\int_{\overline{O}}w(0,y)\, \alpha(dy)=&\ \int_{\overline{O}}\int_{\R^+}\Big[\frac{\partial}{\partial t}w(t,x)+\frac{1}{2}\Delta w(t,x)\Big]\sigma(dt,dx)\\
		&\ -\int_{\overline{O}}\int_{\R^+} w(t,x)\pi(dt,dx).
		\end{align*}
		We define the potential $U_t\in H_0^1(O)$ by
		$$
			\int_{\overline{O}}\frac{1}{2}\nabla h(x)\cdot\nabla  U_t(x) = \int_{\overline{O}}h(x)\sigma(t,x)dx+\int_{\overline{O}}\int_0^th(x)\pi(ds,dx)\ \forall\ h\in H_0^1(O).
		$$
		Now for smooth compactly supported functions $w\in C_c^\infty(\R^+\times \overline{O})$ that vanish on $\partial O$,
		\begin{align}
			&\ \int_{\overline{O}}\int_{\R^+}\frac{1}{2} \big(-\frac{\partial }{\partial t}\nabla w(t,x)\big)\cdot \nabla U_t(x)dtdx \nn\\
			=&\ \int_{\overline{O}}\int_{\R^+}\big(-\frac{\partial}{\partial t} w(t,x)\big)\sigma(t,x)dtdx+ \int_{\overline{O}}\int_{\R^+}\big(-\frac{\partial}{\partial t}w(t,x)\big)\int_0^t\pi(ds,dx)dt\nn\\
			=&\ \int_{\overline{O}} \int_{\R^+} -\frac{1}{2}\nabla w(t,x)\cdot \nabla \sigma(t,x)dtdx+\int_{O} w(0,x)\alpha(x)dx\nn\\
			&\ -\int_{\overline{O}} w(0,y)\pi(\{0\},dy).\nn
		\end{align}
		For any continuous function $g\in C(\R^+\times \overline{O})$ with compact support in $(0,\infty)\times O$, we may select such a $w$ with $-\frac{1}{2}\Delta w(t,x)=g(t,x)$, and after integrating by parts we obtain
		$$
			\int_{\overline{O}}\int_{\R^+} g(t,x)\frac{\partial}{\partial t}U_t(x) dtdx =\int_{\overline{O}}\int_{\R^+} -g(t,x)\sigma(t,x)dtdx.
		$$
		In particular, $\frac{\partial}{\partial t}U_t(x)=-\sigma(t,x)$ for almost every $(t,x)$.
		By the definition of $U$ we also have $-\frac{1}{2}\Delta U_0(x)=\alpha(x)-\pi(\{0\},x)$, with $U_t\rightarrow U_\infty$ as $t\rightarrow \infty$ and $-\frac{1}{2}\Delta U_\infty(x)=\int_{\R^+} \pi(dt,x)$.
		In particular the potential belongs to  $\mathcal{X}$. We note that we have not assumed the non-negativity of $\sigma$ and $\pi$.

		 Using $U_t$ as a test function in (\ref{eqn:Skorokhod_evolution}) we obtain
		\begin{align}
			\int_O U_0(x)\alpha(x)dx=&\ \int_O\int_{\R^+} \Big[\big(-\frac{\partial}{\partial t}U_t(x)\big)\sigma(t,x)+\frac{1}{2}\nabla U_t(x)\cdot \nabla\sigma(t,x)\Big]dtdx\nn\\
			&\ +\int_O\int_{\R^+}U_t(x)\pi(dt,dx)\nn\\
			=&\ \int_O\int_{\R^+} \Big[2\sigma(t,x)^2+\sigma(t,x)\int_0^t\pi(ds,x)\Big]dtdx\label{eqn:identity}\\
			&\ +\int_O\int_{\R^+}U_t(x)\pi(dt,dx).\nn
		\end{align}

		Now we consider two solutions $(\eta_0,\rho_0)$ and ($\eta_1,\rho_1$) satisfying the conditions of the lemma, and we let $\sigma=\eta_0-\eta_1$ and $\pi=\rho_0-\rho_1$.  Then we have that $(\sigma,\pi)$ satisfy (\ref{eqn:Skorokhod_evolution})  with source measure that is identically zero.  Furthermore, $\sigma=0$ a.e.\ in $R$ and $\pi(Z)=0$ for any measurable subset $Z\subset R^c$.  By Lemma \ref{lem:regularity} we have $(\sigma,\pi)\in L^2_\gamma(\R^+;H_0^1(O))\times \mathcal{X}$ and $\pi(\{0\},\cdot)\in L^2(O)$ because $\rho_0(\{0\},\cdot)\leq \mu$, and the same for $(\eta_1,\rho_1)$.

		Assuming that $R$ is a forward-barrier we have $\int_0^t\pi(ds,x)=0$ a.e.\ in $R^c$ so $\int_O\int_{\R^+}\int_0^t\pi(ds,x)\sigma(dt,dx)=0$. In $R$ we have $U_t(x)=U_\infty(x)$ due to $\frac{d}{dt}U_t(x)=-\sigma(t,x)=0$ for $(t,x)\in R$, thus 
		$$
			\int_{\overline{O}}\int_{\R^+}U_t(x)\pi(dt,dx)=\int_O\int_{\R^+}U_\infty(x)\pi(dt,dx)=\frac{1}{2}\int_O |\nabla U_\infty(x)|^2dx.
		$$  
		The initial condition for $U_t$ is $U_0=0$ because if $\rho_1$ is non-zero at $(0,x)$ then it must equal $\alpha(x)$.
		Then the identity (\ref{eqn:identity}) yields that 
		$$
			\int_O\int_{\R^+}|\sigma(t,x)|^2dtdx + \frac{1}{2}\int_O |\nabla U_\infty(x)|^2dx=0
		$$
		 and thus $\eta_0=\eta_1$ and consequentially $\rho_0=\rho_1$.

		In the case of a measurable backward-barrier, we instead have
		$$
			\int_{\overline{O}}\int_{\R^+}U_t(x)\pi(dt,dx)=\int_O\int_{\R^+}U_0(x)\pi(dt,dx)=-\frac{1}{2}\int_{O}|\nabla U_0(x)|^2dx.
		$$
		The value of $U_0$ is $0$ if $s(x)\not=0$, however,	we must consider the possibility that $\pi$ is non-zero at $(0,x)$ in which case $U_0(x)$ solves $-\frac{1}{2}\Delta U_0(x)=\pi(\{0\},x)$, thus is uniquely determined by the value of $\pi$ at $t=0$.  
	\end{proof}
\begin{remark}\label{rem:potential}
	In the above lemma, we consider a potential function $U_t$ associated to the density and stopping distribution involved in the Eulerian formulation, that essentially satisfy  the equation
	\begin{equation}\label{eq:potential}
		-\frac{1}{2}\Delta U_t(x) = \eta(t,x) + \int_0^t\rho(ds,x), 
	\end{equation}
	This potential also satisfies, in case {\normalfont\ref{itm:compounded}}, a quasivariational equation of the form
\begin{equation}\label{CH}
		\min \big\{ \partial_t U_t - \frac{1}{2}\Delta U_t,U_t-U_\nu\}=0,
	\end{equation}
	where $-\frac{1}{2}\Delta U_\nu=\nu$, and $-\frac{1}{2}\Delta U_0=\mu$. 
Note that in one-dimension and for case {\normalfont\ref{itm:compounded}}, one can consider explicitly the potential 
$$
		U_t(x)=\mathbb{E}\big[\big|x-B_{t\wedge\tau}\big|\big], 
	$$
which readily satisfies the above properties; this potential function has already been used in this particular case by \cite{cox2013root} and \cite{gassiat2015root}, and a closely related function was used for the case {\normalfont\ref{itm:discounted}} in one dimension by \cite{deangelis2018}.  	
	We thank A.\ Cox, T.\ De Angelis and M.\ Huesmann for pointing out these papers to us.     
		\end{remark}

	\begin{theorem}\label{thm:pde_characterization}
		Suppose {\normalfont\ref{itm:convex_support}}, {\normalfont\ref{itm:L2_initial}}, {\normalfont\ref{itm:subharmonic_order}}, and {\normalfont\ref{itm:bounded}}. We suppose $(\psi,J_\psi)$ are optimal and that $R_c$ and $R_o$ are defined as above with $\tau_c$ and $\tau_o$ the hitting times in {\normalfont(\ref{eqn:hitting_time})}.

		If {\normalfont \ref{itm:compounded}} then the unique optimal stopping time is given by  $\tau_c$.

 		If {\normalfont \ref{itm:discounted}} then the same holds if the support of $\mu$ and $\nu$ is disjoint.  Otherwise, if $s(x)=0$ in the support of $\mu$ then the optimal stopping time is the unique randomized stopping time that satisfies $\tau_c\leq \tau\leq \tau_o$ and $B_{\tau}\sim \nu$, in particular given $B_0=x$, $\tau=0$ occurs with probability $\nu(x)/\mu(x)$, and otherwise $\tau=\inf\{t>0;(t,B_t)\in R_c\}$.
	\end{theorem}
	\begin{proof}
		We first show that if $\tau$ is optimal then $\tau_c\leq \tau\leq \tau_o$.  From Corollary \ref{cor:martingale_attainment} we have that $J_\psi(\tau,B_\tau)=\psi(B_\tau)$ almost surely, hence $\tau_c\leq \tau$. We now show that 
		$$
			\eta(R_o)=\mathbb{E}\left[\int_0^\tau \mathbf{1}\left\{\begin{array}{ll} 1 & (t,B_t)\in R_o\\ 0 & {\rm otherwise}\end{array}\right\}dt\right]=0,
		$$
		which is equivalent to $\tau\leq \tau_o$ almost surely.  We let $(\eta,\rho)$ be the density for $B_t$ with $t\leq \tau$ and stopping distribution of $(\tau,B_\tau)$ of Proposition \ref{thm:stochastic_embedding}, and we let $\tau^{t,x}\in \mathcal{R}^{t,x}$ be the conditional expectation of $\tau$ given $B_t=x$ as defined in (\ref{eqn:conditional_expectation}).
		From Corollary \ref{cor:martingale_attainment} and Proposition \ref{thm:stochastic_embedding}, we have that
		\begin{align*}
			0=&\ \mathbb{E}\Big[G_\tau-G_{t\wedge \tau}\Big]\\
			=&\ \mathbb{E}\Big[\mathbf{1}\{t\leq \tau\}\Big(\psi(B_\tau)-J_\psi(t,B_t)+\int_t^\tau L(s,B_s)ds\Big)\Big]\\
			=&\ \int_O \left(\mathbb{E}^{t,x}\Big[\psi(B_{\tau^{t,x}})-J_\psi(t,x)+\int_t^{\tau^{t,x}}L(s,B_s)ds\Big]\right)\eta(t,x)dx.
		\end{align*}
		We also have that for each $(t,x)$
		\begin{align}\label{eqn:tau_optimality}
			\mathbb{E}^{t,x}\Big[\psi(B_{\tau^{t,x}})-J_\psi(t,x)+\int_t^{\tau^{t,x}}L(s,B_s)ds\Big]\leq 0,
		\end{align}
		so it follows that for $\eta$-a.e.\ $(t,x)$ equality holds in (\ref{eqn:tau_optimality}), or, in other words, the randomized stopping time $\tau^{t,x}$ is optimal for (\ref{eqn:dynamic_programming_2}).	For $\eta$-a.e.\ $(t,x)\in R_o$ we have that
		$$
			\psi(x)=J_\psi(t,x) = \mathbb{E}^{t,x}\Big[ \psi\big(B_{\tau^{t,x}}\big)-\int_t^{\tau^{t,x}} L\big(s,B_s\big)ds\Big].
		$$
		Now we define $\hat{\tau}^{t,x}= \tau^{t,x} -t + s(x)+\epsilon \in \mathcal{R}^{s(x)+\epsilon,x}$ and the dynamic programing principle implies that if $\tau^{t,x}>t$ then
		\begin{align*}
			J_\psi\big(s(x)+\epsilon,x\big) \geq&\ \mathbb{E}^{s(x)+\epsilon,x}\Big[ \psi\big(B_{\hat{\tau}^{t,x}}\big)-\int_{s(x)+\epsilon}^{\hat{\tau}^{t,x}} L\big(r,B_r\big)dr\Big]\\
			>&\ \mathbb{E}^{t,x}\Big[ \psi\big(B_{\tau^{t,x}}\big)-\int_t^{\tau^{t,x}} L\big(s,B_s\big)ds\Big],
		\end{align*}
		where we have used either \ref{itm:compounded}, $\epsilon>0$ and $t>s(x)+\epsilon$, or we have used \ref{itm:discounted}, $\epsilon<0$ and $t<s(x)+\epsilon$. This contradicts that $J_\psi\big(s(x)+\epsilon,x\big)=\psi(x)$, which implies that $\tau^{t,x}=t$ and completes the claim that $\tau\leq \tau_o$. 

		We have shown that $\rho(R_c)=1$ and $\eta(R_o)=0$.
		 Since $R_o$ and $R_c$ differ by a set of Lebesgue measure zero we also have $\eta(R_c)=0$.

		We next note that for $\tau_c$ there is  a corresponding pair $(\tilde{\eta},\tilde{\rho})$ by  Proposition 
 \ref{thm:stochastic_embedding}.  Furthermore, from the definition of $\tau_c$ we have $\tilde{\eta}(R_o)=0$ and $\tilde{\rho}(R_c)=1$.

		Given \ref{itm:compounded}, Lemma \ref{lem:Eulerian_uniqueness} implies that $\eta(R_c)=0$ and $\rho(R_c)=1$ along with the initial condition uniquely determine $\eta$ and $\rho$ thus $(\tilde{\eta},\tilde{\rho})=(\eta,\rho)$ and $\tau_c$ attains the value $\mathcal{P}_0(\mu,\nu)$ with $B_{\tau_c}\sim \nu$, thus $\tau_c$ is optimal and $\tau=\tau_c$ since $\tau\ge\tau_c$.  

		In the case of \ref{itm:discounted} we similarly apply Lemma \ref{lem:Eulerian_uniqueness} but must also consider the case that $s(x)=0$. In this case exactly $\nu(x)$ mass must stop immediately for the target constraint (\ref{eqn:Eulerian_target}) to be satisfied.  Thus $\tau$ equals the stopping time that stops at time zero with probability $\nu(x)/\mu(x)$ and otherwise is the first positive hitting time of $R_c$, i.e\ $\tau=\inf\{t>0;(t,B_t)\in R_c\}$.
	\end{proof}

	\begin{remark}\label{thm:rigidity}
		It is well known that there is a unique forward (resp.\ backward) barrier yielding a hitting time that embeds the final measure into Brownian motion (see e.g., \cite{beiglboeck2017optimal}).  This clearly implies that any Lagrangian satisfying {\normalfont \ref{itm:compounded}} (resp.\ {\normalfont \ref{itm:discounted}}) lead to the same free boundary and hence optimal stopping time.

		In the case of {\normalfont \ref{itm:stationary}}, i.e.\ that $L(t,x)=\overline{L}(x)$, we may easily construct an optimal dual potential by solving
		$$
			-\frac{1}{2}\Delta \psi(x) = -\overline{L}(x),
		$$
		with Dirichlet boundary conditions.  The value function is then time independent, i.e., $J_{{\psi}}(t,x)={\psi}(x)$ for all time, and every admissible stopping time has the same cost, 
 		$$
 			\mathbb{E}\Big[\int_0^\tau \overline{L}(B_t)dt\Big]=\int_O \psi(z)\nu(dz)-\int_O\psi(y)\mu(dy).
 		$$

		Given a Lagrangian ${L}$ satisfying {\normalfont \ref{itm:compounded}} (resp.\ {\normalfont \ref{itm:discounted}}), and the free boundary $s(x)$ of {\normalfont(\ref{eqn:free_boundary})}, one might expect that the optimal dual potential could be chosen to solve 
 		\begin{align}
 			-\frac{1}{2}\Delta{\psi}(x)=-{L}\big(s(x),x\big),\nn
 		\end{align}
 		with Dirichlet boundary conditions. However, this is not true in general. The function
 		$$
 			{J}(t,x) := \mathbb{E}^{t,x}\Big[{\psi}(B_{\tau_c^{t,x}})-\int_t^{\tau_c^{t,x}} {L}(r,B_r)dr\Big],
 		$$
 		where $\tau_c^{t,x}$ is the first hitting time of $R_c$ given $B_t=x$, satisfies $V_{{\psi}}[{J}](t,x)=0$ whenever $t\not=s(x)$. On the other hand, $({\psi},{J})$ may not be admissible because $V_{{\psi}}[{J}](s(x),x)\not\geq 0$ in the viscosity sense, or similarly $({\psi},{J})$ do not satisfy $V_{{\psi}}[{J}]\geq 0$ weakly.
 	\end{remark}

\section{Free boundary Flow, BSDE and Examples} \label{sec:relationships}
	In this section we consider a few additional aspects of our results as well as  examples. 
\subsection{Free boundary equation}
	It is useful to write down the strong form of the coupled free boundary problem that has arisen in our analysis.  
	The optimality criterion in the Eulerian formulation implies that for optimizers $J(t,x)-\psi(x)=0$ everywhere on the support of ${\rho}$ and that $J(t,x)$ solves the backwards parabolic equation
 	$$
 		-\frac{\partial}{\partial t}J(t,x)=\frac{1}{2}\Delta J (t,x)-L(t,x)
 	$$
 	on the support of $\eta$. 

	For $(t,x)\in R_c^c$ we have the following heat equation for the density $\eta$,
	$$
		\frac{\partial}{\partial t} \eta = \frac{1}{2}\Delta \eta(t,x),
	$$
	with Dirichlet boundary conditions along $(s(x),x)$,
	$$
		\eta(s(x),x)=0,
	$$
	and an initial condition $\eta(0,y)=\mu(y)$. The quantities $s(x)$ and $\mu$ uniquely determine $\eta$ as shown in Lemma \ref{lem:Eulerian_uniqueness} (excepting the case that $R_c$ is a backward barrier and contains $(0,x)$ in the support of $\mu$).  

	To determine $s(x)$ we need to use the constraint that the distribution of stopped particles equals $\nu$.  Note that we must have ${\rho}(dt,dx) = \delta_{s(x)}(dt)\nu(dx)$ so assuming sufficient regularity (especially on $s(x)$), the equation \eqref{eqn:Skorokhod_evolution} gives  (when $L$ is increasing) 
	\begin{align}
		\int_{O}w(s(z),z)\nu(z)dz=&\ \int_{O}\int_{\R^+} w(t,x){\rho}(dt,dx) \nn\\
		=&\ \int_{O}\int_0^{s(x)}\Big[\frac{\partial}{\partial t} w(t,x)+\frac{1}{2}\Delta w(t,x)\Big]\eta(t,x)dtdx\nn\\
		&\ +\int_{O}w(0,y)\mu(y)dy\nn\\
		=&\ -\frac{1}{2}\int_{O}\int_0^{s(x)}\nabla \cdot \Big[w(t,x)\nabla\eta(t,x)\Big]dtdx\nn . 
		\end{align}
		Now apply the Stokes' theorem to get
	\begin{align}
		=&\ \int_{O}\frac{1}{2} \nabla s(x) \cdot \Big[w(s(x),x)\nabla\eta(s(x),x)\Big]dx.\nn
	\end{align} 
	Thus when we consider the flux of stopping-particles we have  the   relation
	$$
		\nu(z) = \pm \frac{1}{2} \nabla s(z) \cdot  \nabla \eta\big(s(z),z\big).	$$
	where the $\pm$ is determined by whether $L$ increases or decreases and $s(x)$ is hit from below or above.  This stopping-rate can be seen to be equivalent to the `two-sided' Stefan problem studied in 1D by \cite{mcconnell1991two}.

\subsection{Martingales and BSDE}
	In comparing  the problems $\mathcal{P}_0(\mu,\nu)$ and $\mathcal{D}_0(\mu,\nu)$ we now have attainment of $\mathcal{D}_0(\mu,\nu)$ at a lower semicontinuous super martingale given by
	$$
		G_t = J_\psi(t,B_t)-\int_0^tL(s,B_s)ds.
	$$
	Indeed this martingale is the Snell envelope of the process $Y_t=\psi(B_t)-\int_0^tL(s,B_s)ds$
	as $G_{t\wedge \tau}$ is a martingale for $\tau$ the Snell hitting time, which agrees with the optimal stopping in the case of \ref{itm:compounded} and \ref{itm:discounted} (modulo behavior at $t=0$).

	To compare with the backward stochastic differential equations as studied in \cite{el1997reflected}, \cite{touzi2012optimal}, \cite{pham1997optimal} and others, we assume that $J_\psi\in C_{-\gamma}^{1,2}(\R^+\times \overline{O})$ for 
	$J_\psi$ satisfying  
	  $V_\psi[J_\psi]\geq 0$.
	  Consider $Z_t=J_\psi(t,B_t)$, the random vector $P_t = \nabla J_\psi(t,B_t)$, and the random matrix $Q_t=\nabla^2J_\psi(t,B_t)$. We have $Z_t\geq \psi(B_t)$ and Ito's formula shows that 
	\begin{align}
		P_t =&\ P_\tau +\int_t^\tau \Big[\nabla \frac{\partial}{\partial t}J_\psi(s,B_s)+\frac{1}{2}\nabla\Delta J_\psi(s,B_s)\Big]ds+\int_t^\tau \nabla^2J_\psi(s,B_s) dB_s\nn \\
		=&\ P_\tau+\int_t^\tau \nabla L(s,B_s) ds + \int_t^\tau Q_sdB_s.\nn
	\end{align}
	In other words, $(Z,P,Q)$ solve the backward stochastic differential equation
	\begin{align}\label{eqn:bsde}
		-dP_t=&\ \nabla L(t,B_t)dt+Q_tdB_t, \\
		-dZ_t =&\ -L(t,B_t)dt-P_t\cdot dB_t,
	\end{align}
	with $Z_\tau =\psi(B_\tau)$ and $P_\tau=\nabla \psi(B_\tau)$, along with $Z_t\geq \psi(B_t)$ for $t\leq \tau$.

	\begin{remark}
		This remark leads to a third dual formulation (that seems to require slightly more regularity of $L$):
		$$
		\hbox{Maximize} \quad	\int_{\overline{O}}\psi(z)\nu(dz)-\mathbb{E}\big[Z_0\big]
		$$
		subject to $(P,Q,Z)$ solve {\normalfont(\ref{eqn:bsde})} for $t\leq \tau$ with $Z_\tau =\psi(B_\tau)$, $P_\tau=\nabla \psi(B_\tau)$, and $Z_t\geq \psi(B_t)$ for $t\leq \tau$.
	\end{remark}

	\subsection{Examples}
	The simplest examples of optimal stopping times (although not fitting in our compactly supported setting) occur when $\nu(z)=\tilde{\mu}(t_1,z)$ for $\tilde{\mu}$ the solution to the heat equation with $\tilde{\mu}(0,y)=\mu(y)$ and $t_1$ is a constant. 
	Since $\tau=t_1$ is the hitting-time of a forward-barrier, by the results of Section \ref{sec:hitting_times} (modulo the noncompactness of the supports), it follows that $\tau$ is optimal if $L(t,x)$ is strictly increasing.
	 In the case that $L(t,x)=L(t)$ we can easily compute the optimal cost
	$$
		\mathcal{P}_0(\mu,\nu)=\int_0^{t_1}L(t)dt.
	$$
	To compare with the dual problem, we let $\psi$ solve
	$$
		-\frac{1}{2}\Delta \psi(z)=-L(t_1),
	$$
	which has a bounded below solution $\psi(z)=L(t_1)Q(z)$, where $Q(z)$ is any nonnegative quadratic with $\nabla^2Q(z)=\bs\Sigma$ and $\frac{1}{2}{\rm trace}(\bs\Sigma) =1$ (note that the difference of any two such solutions differs only by a harmonic polynomial).  The value function is solved by
	$$
		J_\psi(t,x)=\psi(x) +f(t)
	$$
	where $f$ solves
	$$
		-f'(t) + L(t) = L(t_1)
	$$
	given by
	$$
		f(t)= L(t_1)(t_1-t)-\int_t^{t_1} L(s)ds.
	$$
	We let $\bs\Sigma_\mu$ denote the covariance of $\mu$ and $\bs\Sigma_\nu$ denote the covariance of $\nu$.  It follows from the heat equation that $\bs\Sigma_\nu-\bs\Sigma_\mu = \frac{1}{2}t_1\mb{1}$.  Then the dual value is
	\begin{align}
		\mathcal{D}_1(\mu,\nu)&=\ \int_{\R^d}\psi(z)\nu(dz)-\int_{\R^d}J_\psi(0,y)\mu(dy)\nn\\
		=&\  L(t_1)\bs\Sigma \cdot (\bs\Sigma_\nu - \bs\Sigma_\mu) -L(t_1)t_1+\int_0^{t_1} L(t)dt= \int_0^{t_1} L(t)dt, \nn
	\end{align}
	demonstrating the duality principle.
	For the case that $L(t)$ decreases, we would first
	subtract the overlap $\mu \wedge \nu$ and let $\mu^+=\mu- \mu\wedge\nu$ and $\nu^+=\nu-\mu\wedge\nu$, 
	and all the mass except $\mu^+$ would stop immediately.  Then the mass of $\mu^+$ would be transported to $\nu^+$ along some free boundary $s(x)$ with $s(\alpha)=0$ where $\alpha$ is any point such that $\mu(\alpha)=\nu(\alpha)$.

	We conclude with some figures illustrating the simulation of general distributions and Lagrangians that will be explored in a further numerical study.

	We have discretized the domain as the integers between $0$ and $39$.  The initial distribution $\mu$ is the uniform measure on the integers between $16$ and $23$, and the target measure $\nu$ is proportional to $|\sin(\frac{\pi x}{13})|$.  We perform the convex optimization of the dual problem to approximate the free boundary yielding the simulation for the increasing (I.) and decreasing (II.) cases, see Figure \ref{fig:inc_dec}.
	\begin{figure}
	\centering
	$
     \  I.\  \begin{array}{cc}
      \includegraphics[height=13.0em]{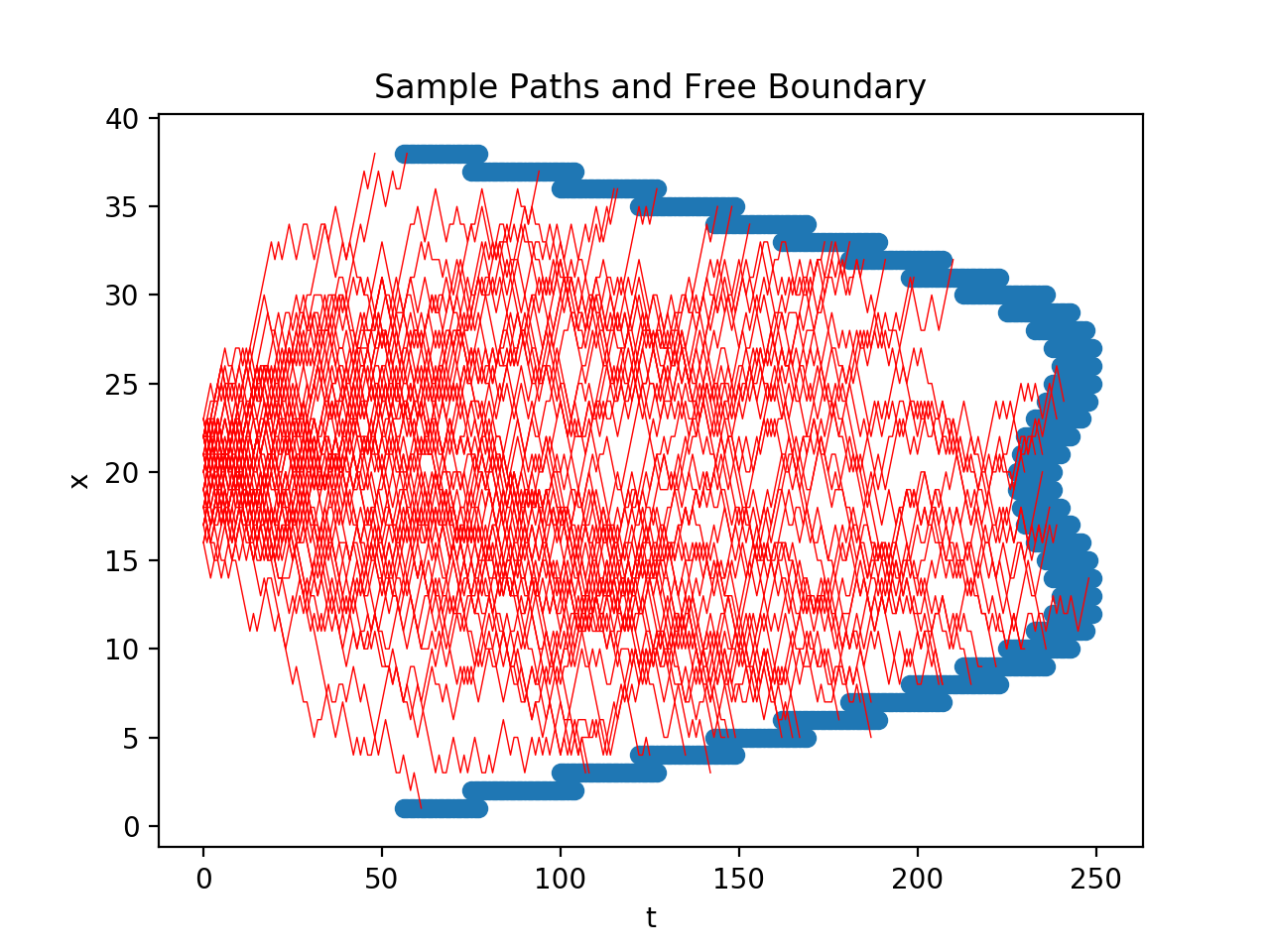} & \includegraphics[height=13.0em]{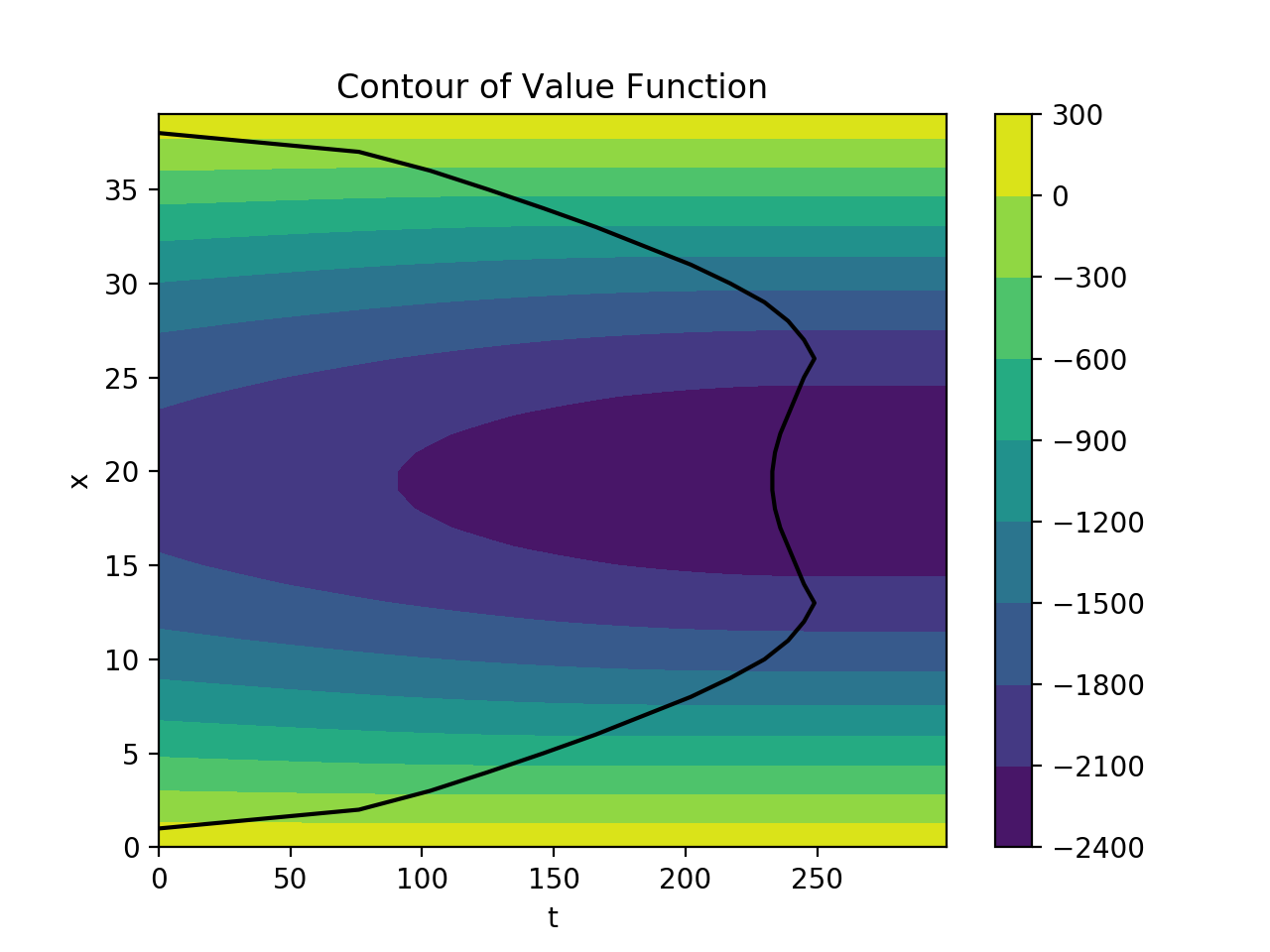}
      \end{array}
    $\\
    \ \\
    $
      II. \begin{array}{cc}
      \includegraphics[height=13.0em]{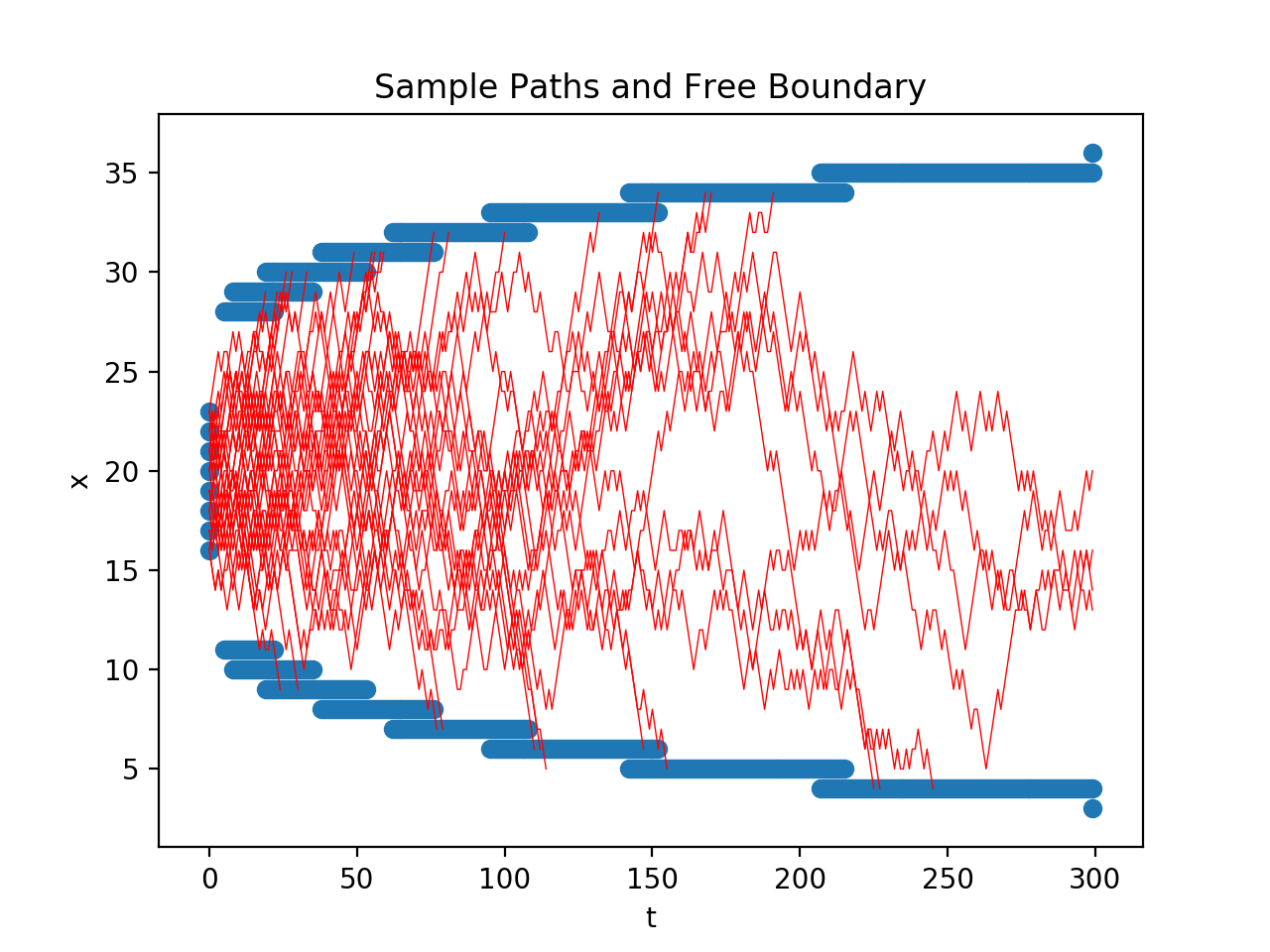} & \includegraphics[height=13.0em]{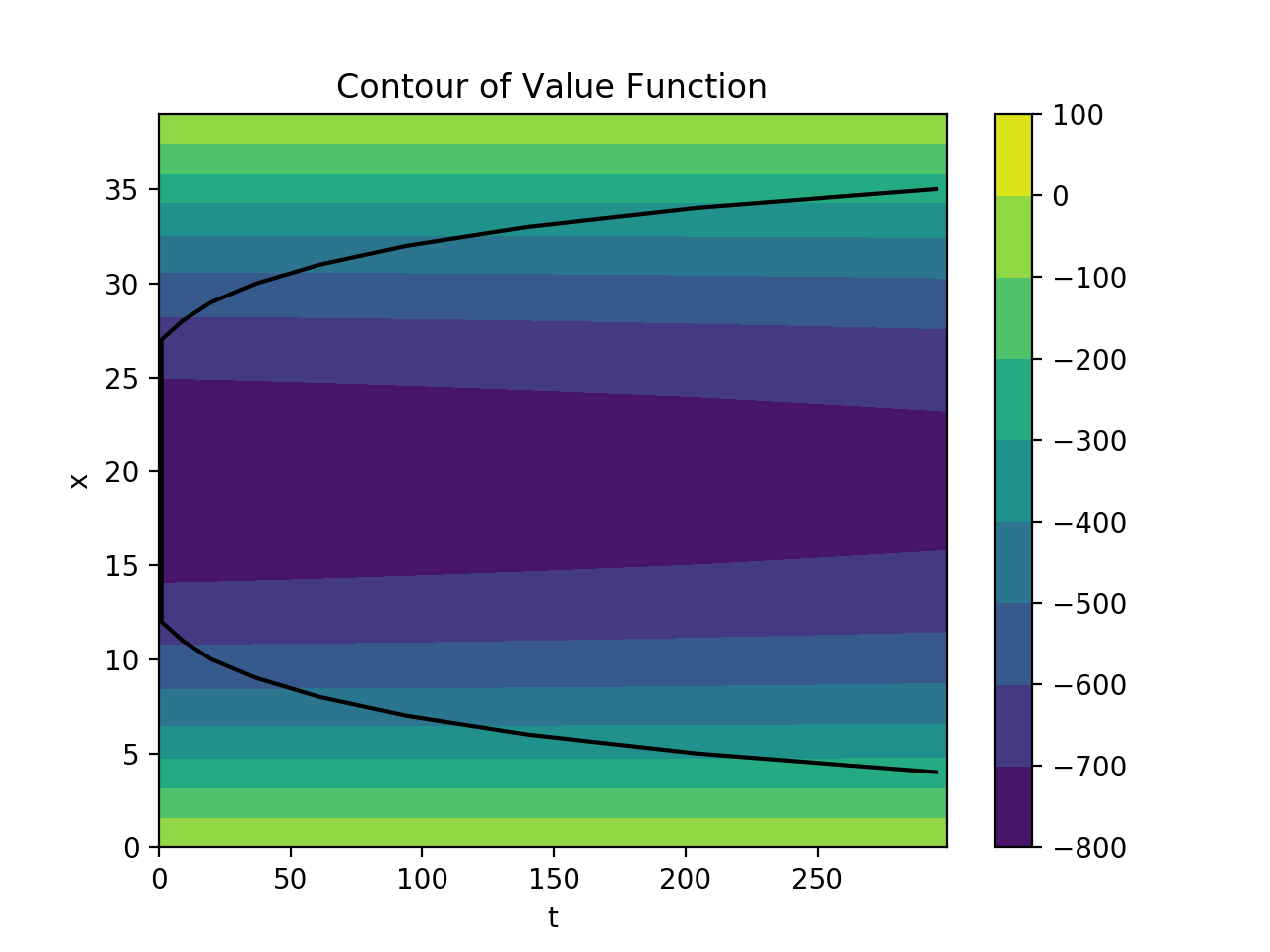}
      \end{array}
    $
    \caption{\label{fig:inc_dec}On the left are simulated trajectories of a random walk approximating Brownian motion and the stopped distribution. On the right we have contours of the value function $J_\psi$ along with the free boundary $s(x)$.}
    \end{figure}
    The solution method does not depend on the property of the Lagrangian thus in Figure 2 we also simulate the Skorokhod embedding for an oscillating Lagrangian proportional to $1-\cos(20\pi t)$.\\
	
	\begin{figure}
	\centering
	$
    \begin{array}{cc}
      \includegraphics[height=15em, trim={0.8cm 0 0.4cm 0}, clip]{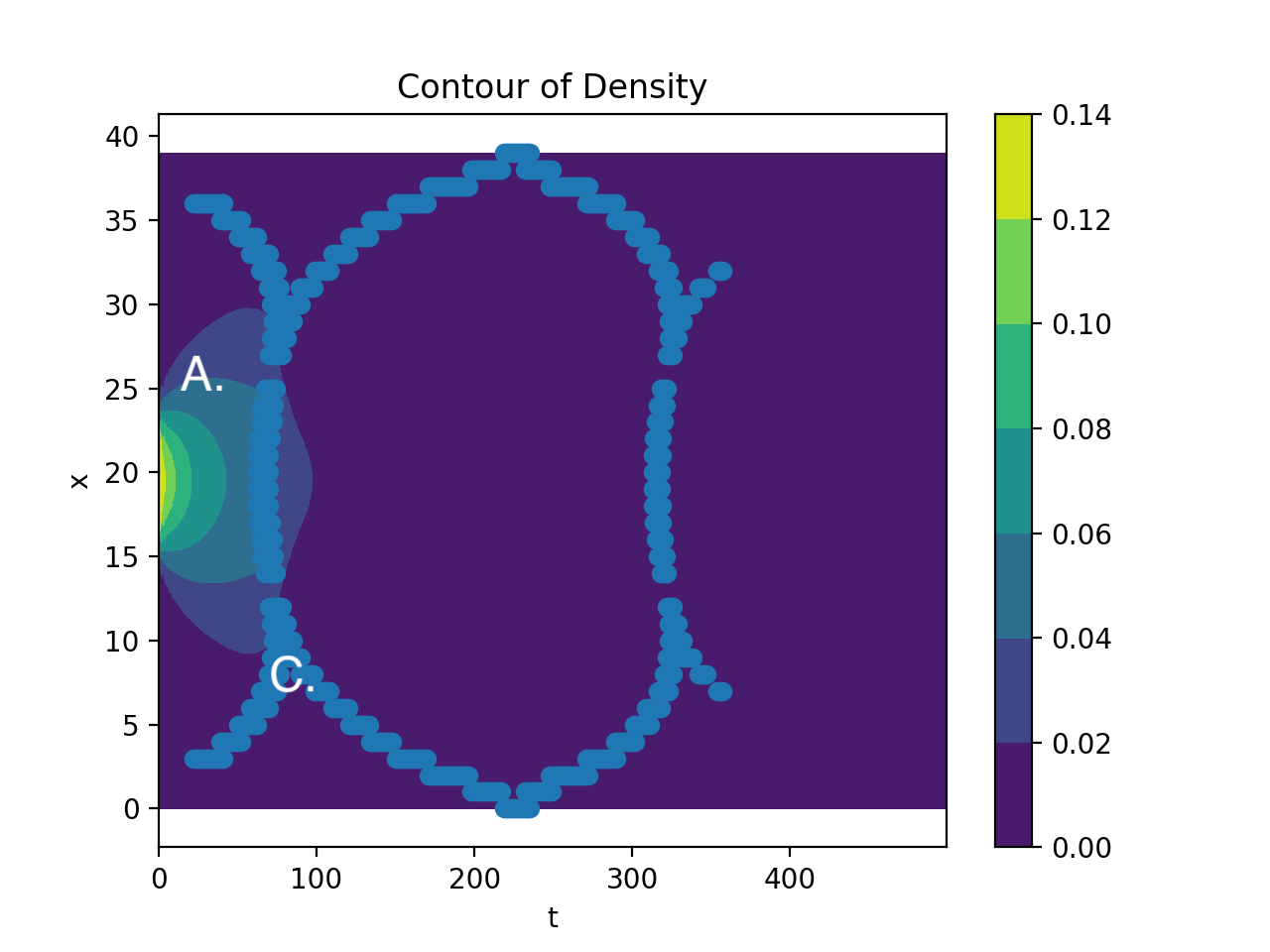} & \hspace{-2.5em} \includegraphics[height=15em, trim={0.8cm 0 1.8cm 0}, clip]{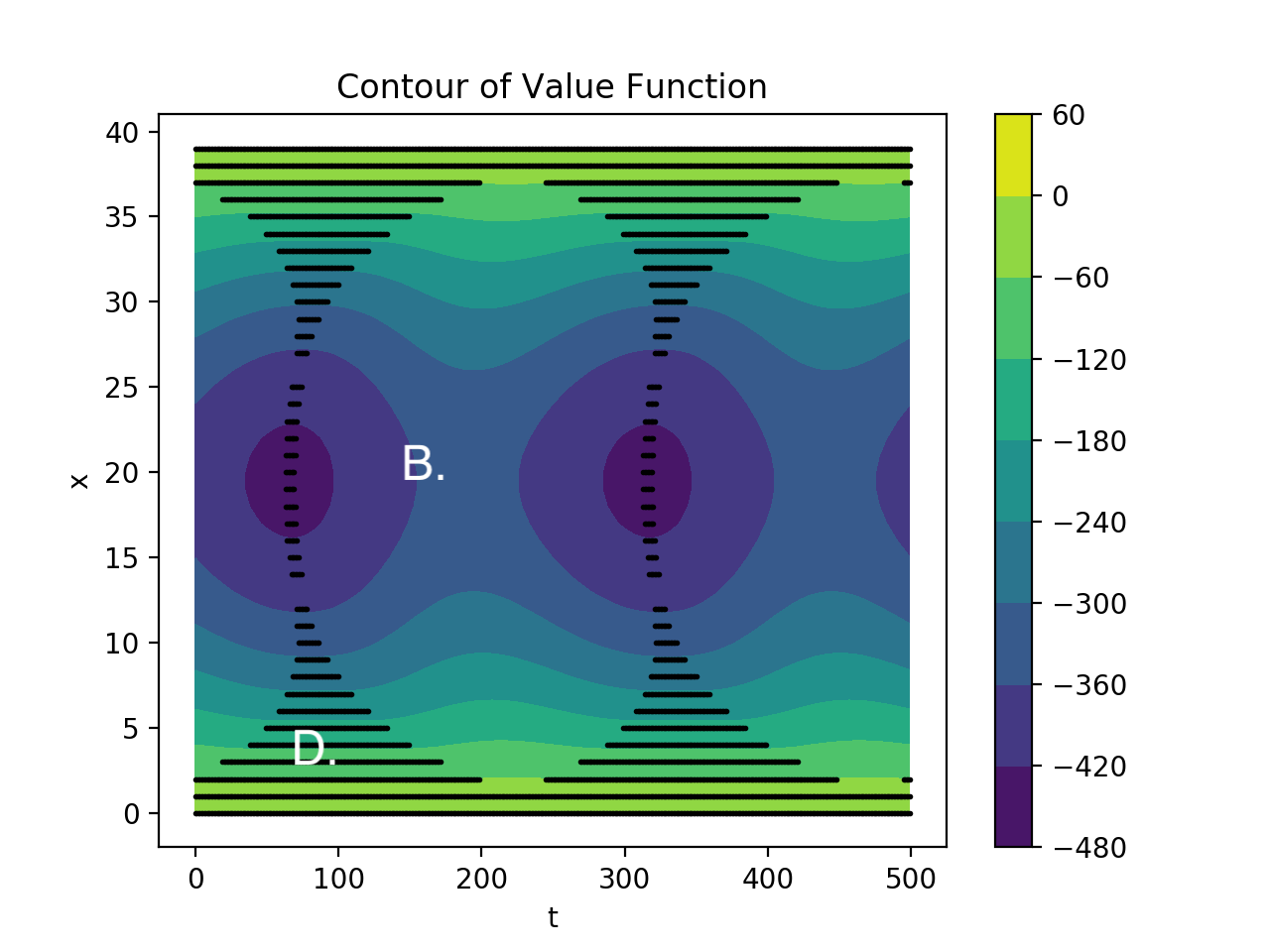}
    \end{array}
    $
    \caption{\label{fig:generic}The contours of the density and points for the stopping distribution are plotted on the left, with the continuation domain $A$ and stopping region $C$. On the right we have contours of the value function $J_\psi$ along with lines marking the coincidence region $D$ and the region where $J_\psi>\psi$, $B$.}
    \end{figure}

\appendix

\section{}\label{S:appendix}
\subsection{Weak duality}
	\begin{theorem}\label{thm:probabilistic_weak_duality}
		The primal and dual problem have the same value $\mathcal{P}_0(\mu,\nu)=\mathcal{D}_0(\mu,\nu)$.
	\end{theorem}
	\begin{proof}
		We exploit a duality between progressively measurable continuous supermartingales with exponential growth, $G\in \mathcal{K}^+_{-\gamma}$, and randomized stopping times with exponential decay $\tau\sim \alpha \in \mathcal{R}$ with $\E[\e{\gamma \tau}]<\infty$.  We define $\mathcal{A}_{-\gamma}\subset C(\R^+\times\Omega)$ where $\Omega$ has the topology of uniform convergence, $\sup_{(t,\omega)\in \R^+\times \Omega}\e{-\gamma t}A_t<\infty$, and $\e{-\gamma t}A_t\rightarrow 0$ uniformly on compact sets. Furthermore, we assume $A\in  \mathcal{A}_{-\gamma}$ is progressively measurable, i.e.\ for $s\in [0,t]$, $s\mapsto A$ is $\mathcal{F}_t$ measurable. We consider elements of the dual space $\mathcal{A}_{-\gamma}^*$ with $\mathbb{P}$ as their $\Omega$ marginal, what satisfy that the disintegration $\beta:\Omega\rightarrow \mathcal{M}(\R^+)$ has $\beta([0,t])$ is $\mathcal{F}_t$ measurable for any $t\geq 0$.  
		The randomized stopping times, $\alpha\in \mathcal{R}$, with exponential decay are such elements that also satisfy $\alpha\geq 0$ and $\alpha(\R^+)=1$.

		The dual relationship is given by, for $A\in \mathcal{A}_{-\gamma}$ and $\tau\sim \alpha\in \mathcal{R}$ with $\E[\e{\gamma \tau}]<\infty$,
		$$
			\mathbb{E}\big[A_\tau\big]:=\mathbb{E} \Big[ \int_{\R^+} A_t\alpha(dt)\Big].
		$$

		We define $\Theta:\mathcal{A}_{-\gamma}\rightarrow \R \cup \{+\infty\}$ as
		$$
			\Theta(f)=\begin{cases} 0 &\ A_t\geq -\int_0^tL(s,B_s)ds,\ \forall\ (t,\omega),\\ +\infty &\ {\rm otherwise}
			\end{cases}
		$$
		so that the Legendre transform is, for $\tau\sim \alpha\in \mathcal{R}$ with $\E[\e{\gamma \tau}]<\infty$,
		$$
			\Theta^*(-\alpha)=\sup_{A\in \mathcal{A}_{-\gamma}}\Big\{-\mathbb{E}\big[A_\tau\big]-\Theta(A)\Big\}=  \mathbb{E}\Big[\int_0^\tau L(t,B_t\big)dt\Big].
		$$
		For general $\beta\in \mathcal{A}_{-\gamma}^*$, the Legendre transform is $+\infty$ if $\beta\not\geq 0$, and is the pairing of $\beta$ and $\int_0^tL(s,B_s)ds$ otherwise. 

		We define $\Xi$ to be
		\begin{align}
			\Xi(A) = \inf_{\psi\in C(\overline{O}),G\in \mathcal{K}^+_{-\gamma}}\Big\{&-\int_{\overline{O}}\psi(z)\nu(dz)+\mathbb{E}\big[G_0\big];\ G_t-\psi(B_t)\geq A_t\Big\}.\nn
		\end{align}
		Then we calculate the Legendre transform for $\beta\in \mathcal{A}_{-\gamma}^*$:
		\begin{align}
			\Xi^*(\beta) =&\ \sup_{A\in \mathcal{A}_{-\gamma}} \Big\{\int_{\Omega}\int_{\R^+}A_t(\omega)\beta(dt,d\omega)-\Xi(A)\Big\}\nn\\
			=&\ \sup_{\psi\in C(\overline{O}),G\in \mathcal{K}^+_{-\gamma}} \Big\{\int_{\Omega}\int_{\R^+}\left[G_t(\omega)-\psi\big(B_t(\omega)\big)\right]\beta(dt,d\omega)-\mathbb{E}[G_0]+\int_{\R^d}\psi(z)\nu(dz)\Big\}.\nn
		\end{align}
		For $\beta\in \mathcal{A}_{-\gamma}^*$ with $\beta\geq 0$, then
		$$
			\sup_{G\in \mathcal{K}_{-\gamma}^+}\Big\{\int_{\Omega}\int_{\R^+} G_t(\omega)\beta(dt,d\omega)-\E[G_0]\Big\}=0
		$$
		if and only if $\beta\in \mathcal{R}$, and it is $+\infty$ otherwise. Since and $\tilde{G}\in C(\Omega)$ defines $G\in \mathcal{K}_{-\gamma}^+$ by $G_t(\omega):=\tilde G(\omega)$, we have that the marginal on $\Omega$ of $\beta$ is $\mathbb{P}$, and the growth condition follows easily from multiplying by $\e{-\gamma t}$.  The supermartingale property and the optional stopping theorem then implies that
		$$
			\int_{\Omega}\int_{\R^+}G_t(\omega)\beta(dt,d\omega) =\E \big[G_\tau\big]\leq \E\big[G_0\big].
		$$

		By definition of $\Xi$ we have that $\mathcal{D}_1(\mu,\nu)=-\Xi(A^L)$ where $A_t^L=-\int_0^tL(s,B_s)ds$.
		Similarly, the primal problem is
		$$
			\sup_{\tau\in \mathcal{R}_{\gamma}} \big\{-\Theta^*(-\alpha)-\Xi^*(\alpha)\big\}.
		$$
		Fenchel-Rockafellar applies after showing there is $A\in \mathcal{A}_{-\gamma}$ with $\Xi[A],\Theta[A]<\infty$ and $\Theta$ continuous at $A$, but this obviously holds, for example with $A^L+\e{\gamma t}$.
	\end{proof}

\subsection{Variational inequalities} 
\begin{definition}\label{def:viscosity_solution}
	Suppose $\psi$ and $L$ are continuous, and $J$ is lower semicontinuous.  Then we say that $V_\psi[J]\geq 0$ {\em in the sense of viscosity}  (viscosity supersolution)  if whenever a smooth function $w$ touches $J$ from below at $(t,x)$, i.e.\ $w(t,x)=J(t,x)$ and $w(s,y)\leq J(s,y)$ for all $(s,y)\in \R^+\times \overline{O}$, then  $V_\psi[w](t,x)\geq 0$.

	Similarly, if $J$ is upper semicontinuous we say that $V_\psi[J]\leq 0$ {\em in the sense of viscosity}  (viscosity subsolution)  if whenever $w$ touches $J$ from above at $(t,x)$ then  $V_\psi[w](t,x)\leq 0$.

	For $J$ continuous, we say that $V_\psi[J]=0$ {\em in the sense of viscosity}, if $J$ is both a subsolution and supersolution. 
\end{definition}

\begin{proposition}\label{prop:lower_semicontinuity}
	Suppose $\psi\in H_0^1(O)$, $\psi\leq 0$ and $-\Delta \psi$ is bounded below.  Then there is a lower semicontinuous representative of $\psi$.  

	Similarly, suppose $J \in \mathcal{X}$, $J\leq 0$,
	and $-\frac{\partial}{\partial t}J-\frac{1}{2}\Delta J$ is bounded below, then there is a lower semicontinuous representative of $J$. 
\end{proposition}
\begin{proof}
	Essentially this result follows from approximating from below by convolution with smooth cut-off functions and the mean value property of supersolutions.
	The limit of these approximations is lower semicontinuous. A classic resource is \cite{lewy1970smoothness} and \cite{ziemer1988regularity} does the parabolic version.  See also \cite{Sylvestre2015}. 
\end{proof}

\begin{proposition}\label{prop:Perron}
	Suppose that $L\in C_{-\gamma}(\R^+\times \overline{O})$, $\psi\in C(\overline{O})$, then there is a unique viscosity solution $J_\psi\in C_{-\gamma}(\R^+\times \overline{O})$ to $V_\psi[J]=0$ given by
	$$
		J_\psi(t,x) = \inf\{\phi(t,x)\in C_{-\gamma}^{1,2}(\R^+\times\overline{O});\ V_\psi[\phi]\geq 0\}.
	$$
\end{proposition}
\begin{proof}
	Existence of a continuous solution can be found in \cite{el1997reflected}.  The proof that such a solution may be approximated above by a smooth supersolution can be done by regularization, see for instance \cite{krylov2000rate}, \cite{jakobsen2003rate}, \cite{barles2007error}. 
\end{proof}

\begin{definition}\label{def:variational_inequality}
	Given $L\in C_{-\gamma}(\R^+\times \overline{O})$, $\psi\in H_0^1(O)$ and $J\in \mathcal{X}$ we say that $V_\psi[J]\geq 0$ weakly, if $J(t,x)\geq \psi(x)$ for almost every $(t,x)$ and 
	\begin{align}\label{eqn:weak-variational}
		\int_{\R^+}\int_{\overline{O}} \Big[-w(t,x)\frac{\partial}{\partial t}J(t,x)+\frac{1}{2}\nabla w(t,x)\cdot \nabla J(t,x)+w(t,x)L(t,x)\Big]dxdt\geq 0
	\end{align}
	for any $w\in L_\gamma(\R^+;H_0^1(O))$ with $w\geq 0$.

	We say that $V_\psi[J]=0$ weakly if  $V_\psi[J]\ge 0$ in a weak sense and the  inequality \eqref{eqn:weak-variational}  
	 holds 
	whenever $J-\psi+w\geq 0$.

	We say that $J\in \mathcal{X}$ is the minimal weak solution if $V_\psi[J]=0$ weakly and for every weak solution to $V_\psi[\tilde{J}]=0$, $J(t,x)\leq \tilde{J}(t,x)$ almost everywhere.
\end{definition}

Our definition of weak solution and minimal weak solution is equivalent to the definitions in \cite{bensoussan2011applications} (Chapter 3, Section 2) where they have used negative $\psi$ and $J$ and thus consider maximal weak solutions.  Note also that they use test functions that are equivalent to $v(t,x)=-J(t,x)-w(t,x)$, i.e.\ satisfy $v(t,x)\leq -\psi(x)$.

\begin{proposition}\label{prop:viscosity_variational_equivalence}
	Suppose {\normalfont \ref{itm:continuous}} and $\psi\in LSC(\overline{O})\cap H_0^1(O)$ with $\psi\leq 0$.  Then there is a unique function $J_\psi\in \mathcal{X}$, $J_\psi\leq 0$, that is the minimal weak solution to $V_\psi[J_\psi]=0$, and there is a constant $C$ such that 
	$$
		\|J_\psi\|_{\mathcal{X}}\leq C\|\psi\|_{H_0^1}.
	$$

	Furthermore, $J_\psi$ is the value function for the optimal stopping problem with terminal reward $\psi$ and satisfies
	$$
		J_\psi(t,x) = \sup_{\tau\in \mathcal{R}^{t,x}}\Big\{\mathbb{E}^{t,x}\Big[\psi(B_\tau)-\int_t^\tau L(s,B_s)ds\Big]\Big\}
	$$
\end{proposition}
\begin{proof}
	See Theorem 4.6 in Chapter 3 of \cite{bensoussan2011applications}.

	For the uniform bound, we use the test function $w(t,x)= \e{-2\gamma t}(\psi(x)-J_\psi(t,x))$, which satisfies $J-\psi+w\geq 0$, and thus (\ref{eqn:weak-variational}) implies
	\begin{align}
		&\ \frac{1}{2}\int_{O}|J_\psi(0,\cdot)|^2dx+\frac{1}{2}\int_{\R^+}\int_{{O}}\e{-2\gamma t}|\nabla J_\psi(t,x)|^2dxdt\nn\\
		\leq&\ \int_{{O}}\psi(x)J_\psi(0,x)dx + \int_{\R^+}\int_{\overline{O}} \gamma \e{-2\gamma t} (J_\psi(t,x)-2\psi(x))J(t,x)dxdt\nn\\
		&\ +\int_{\R^+}\int_{{O}}\e{-2\gamma t} \Big( \nabla \psi(x)\cdot \nabla J_\psi(t,x)+ (\psi(x)-J_\psi(t,x))L(t,x)\Big)dxdt\nn\\
		\leq&\ C\|\psi\|_{H_0^1(O)}\sqrt{\int_{O}|J_\psi(0,\cdot)|^2dx+\int_{\R^+}\int_{{O}}\e{-2\gamma t}|\nabla J_\psi(t,x)|^2dxdt}.\nn
	\end{align}
\end{proof}

\bibliography{Skorokhod}
\bibliographystyle{plain}

\end{document}